\documentclass[onefignum,onetabnum,pagebackref,dvipsnames]{siamart220329}

\usepackage{mathtools, amsfonts, amssymb, mathrsfs}
\usepackage{enumerate,enumitem}
\usepackage{subcaption}
\usepackage{soul}
\usepackage{algorithm, algpseudocode, algorithmicx}
\usepackage{comment}
\usepackage{xspace}
\usepackage{mleftright}

\renewcommand*{\backref}[1]{}
\renewcommand*{\backrefalt}[4]{%
	\ifcase #1 %
	(No citations.)
	\or
	(Cited on page #2.)
	\else
	(Cited on pages #2.)
	\fi
}

\usepackage[most]{tcolorbox}

\algrenewcommand\algorithmiccomment[1]{\hfill\textcolor{CornflowerBlue}{$\triangleright$ #1}}

\newcommand{\real}{\mathbb{R}}

\DeclareMathOperator{\diag}{diag}

\newcommand{\mat}[1]{\boldsymbol{#1}}
\renewcommand{\vec}[1]{\boldsymbol{#1}}

\newcommand{\norm}[1]{\mleft\| #1 \mright\|}

\newcommand{\QR}{\textsf{QR}\xspace}
\newcommand{\ULV}{\textsf{ULV}\xspace}
\newcommand{\LU}{\textsf{LU}\xspace}

\newcommand{\expmat}[1]{\begin{bmatrix} #1 \end{bmatrix}}
\newcommand{\twobytwo}[4]{\expmat{#1 & #2 \\ #3 & #4}}
\newcommand{\twobyone}[2]{\expmat{#1 \\ #2}}

\newcommand{\Id}{\mathbf{I}}

\newcommand{\order}{\mathcal{O}}

\DeclareMathOperator{\fl}{fl}

\DeclareMathOperator*{\argmin}{argmin}
\newcommand{\set}[1]{\mathsf{#1}}
\newcommand{\e}{\mathrm{e}}

\renewcommand{\hat}[1]{\widehat{#1}}
\renewcommand{\tilde}[1]{\widetilde{#1}}
\newcommand{\actionbox}[1]{\begin{tcolorbox}[colback=white,colframe=black,width=\columnwidth,boxsep=5pt,arc=4pt]
    #1
\end{tcolorbox}}

\DeclareMathOperator{\cond}{cond}
\DeclareMathOperator{\err}{err}
\newcommand{\myparagraph}[1]{\vspace{0.3em}

\textit{\textbf{#1}}}
\usepackage{amsfonts}
\usepackage{graphicx}
\usepackage{epstopdf}
\ifpdf
  \DeclareGraphicsExtensions{.eps,.pdf,.png,.jpg}
\else
  \DeclareGraphicsExtensions{.eps}
\fi

\usepackage{cite}

\newsiamremark{remark}{Remark}
\newsiamremark{hypothesis}{Hypothesis}
\crefname{hypothesis}{Hypothesis}{Hypotheses}
\newsiamthm{claim}{Claim}
\newsiamthm{inftheorem}{Informal Theorem}
\newsiamthm{fact}{Fact}

\headers{Preconditioning the Normal Equations}{E. N. Epperly, A. Greenbaum, and Y. Nakatsukasa}

\title{Stable Algorithms for General Linear Systems by Preconditioning the Normal Equations\thanks{Date: March 4, 2025.
\funding{ENE acknowledges support from the DOE computational science graduate fellowship under grant DE-SC0021110 and, under aegis of Joel Tropp, the Office of Naval Research through BRC Award N000142412223 and the Carver Mead New Horizons Fund.
YN is supported by the EPSRC grant EP/Y030990/1.
}}}

\author{Ethan N. Epperly\thanks{Department of Computing and Mathematical Sciences, California Institute of Technology, Pasadena, CA 91125 USA (\email{eepperly@caltech.edu}, \url{https://ethanepperly.com}).} \and Anne Greenbaum\thanks{Department of Applied Mathematics, University of Washington, Seattle, WA 98195 USA (\email{greenbau@uw.edu}).} \and Yuji Nakatsukasa\thanks{Mathematical Institute, University of Oxford, Oxford, OX2 6GG, UK (\email{nakatsukasa@maths.ox.ac.uk}).}}

\usepackage{amsopn}

\usepackage[normalem]{ulem}

\newcommand{\ignore}[1]{}

\ifpdf
\hypersetup{
  pdftitle={Stable algorithms for general linear systems by preconditioning the normal equations},
  pdfauthor={E. N. Epperly, A. Greenbaum, and Y. Nakatsukasa}
}
\fi

\begin{document}

\maketitle

\begin{abstract}
    This paper studies the solution of nonsymmetric linear systems by preconditioned Krylov methods based on the normal equations, LSQR in particular.
    On some examples, preconditioned LSQR is seen to produce errors many orders of magnitude larger than classical direct methods; this paper demonstrates that the attainable accuracy of preconditioned LSQR can be greatly improved by applying iterative refinement or restarting when the accuracy stalls.
    This observation is supported by rigorous backward error analysis.
    This paper also provides a discussion of the relative merits of GMRES and LSQR for solving nonsymmetric linear systems, demonstrates stability for left-preconditioned LSQR \emph{without iterative refinement}, and shows that iterative refinement can also improve the accuracy of preconditioned conjugate gradient.
\end{abstract}

\begin{keywords}
    linear system, preconditioning, numerical stability, backward stability, LSQR
\end{keywords}

\begin{MSCcodes}
    65F08, 65F10
\end{MSCcodes}

\section{Introduction}

As Saad notes in his famous textbook, ``finding a good preconditioner to solve a given sparse linear system is often viewed as a combination of art and science'' \cite[p.~283]{Saa03}. 
While this quote speaks to preconditioning sparse linear systems, preconditioning for dense systems is an equally challenging task.

Since Saad wrote those words, the ``science'' of preconditioners has made significant strides.
Many of the advances involve randomization.
Preconditioners for graph Laplacian systems have grown from their roots in theoretical computer science \cite{ST04,KOSZ13,KS16} into practical tools for scientific computing \cite{CLB21,GKS23,LB12}.
These methods extend to general symmetric (weakly) diagonally dominant linear systems and, in theory at least, nonsymmetric systems as well \cite{Kyn17,FM18}.
Preconditioners based on randomized dimensionality reduction yield near-perfect preconditioners for highly overdetermined least-squares problems \cite{RT08,AMT10}.
A recent series of works have developed preconditioners with rigorous guarantees for solving symmetric positive definite systems with a small number of large eigenvalues \cite{FTU23,DEF+23,DMY24,COCF16}.
Even ``classical'' preconditioning strategies like algebraic multigrid have seen significant development, including the development of methods for difficult convection--diffusion problems with guaranteed convergence rates \cite{Not12}.

Given these exciting developments, one can imagine a future where preconditioned iterative solvers form a more integral part of programming environments such as MATLAB, numpy, and Julia.
One can imagine, for instance, that when a user of MATLAB types \texttt{A{\textbackslash}b}, the software automatically checks whether $\mat{A}$ is diagonally dominant and, if so, invokes a preconditioned iterative solver with a high-quality preconditioner computed by a randomized Cholesky or \LU decomposition.
To see this vision come to pass, or more modestly, to expand the scope of applications where we can safely use preconditioned iterative methods, it is important that preconditioned iterative methods are just as accurate and reliable as their direct counterparts.
This paper addresses this need, developing \emph{numerically stable} ways of using preconditioners to solve nonsymmetric linear systems.

\subsection{Our approach: Preconditioning the normal equations}

This paper considers the problem of solving a linear system of equations $\mat{A}\vec{x} = \vec{b}$ for a nonsingular matrix $\mat{A} \in \real^{n\times n}$ and vector $\vec{b} \in \real^n$.
We assume access to an excellent preconditioner $\mat{P}\in\real^{n\times n}$ which controls the condition number of $\mat{A}$ up to an absolute constant:
\begin{equation} \label{eq:good-prec}
    \cond(\mat{A}\mat{P}^{-1}) \le C.
\end{equation}
For this article, $\cond(\mat{B}) = \sigma_{\rm max}(\mat{B}) / \sigma_{\min}(\mat{B})$ denotes the spectral norm condition number.
Constructions for preconditioners were discussed above and will be discussed more in \cref{sec:why-preconditioned-lsqr}.

This article will investigate solving $\mat{A}\vec{x} = \vec{b}$ by \emph{preconditioned LSQR} (PLSQR), which is algebraically equivalent to conjugate gradient (CG) applied to a preconditioned version of the normal equations \cref{eq:pre-normal}.
Our focus on LSQR instead of GMRES is discussed below in \cref{sec:lsqr}.
Without modification, PLSQR is not stable.
To see this, we generate a random matrix $\mat{A}$ with condition number $\cond(\mat{A}) = 10^{10}$ and near-perfect preconditioner $\mat{P}^{-1}$ satisfying $\cond(\mat{A}\mat{P}^{-1}) = 4$; see \cref{sec:experiment-details} for details.
The results of applying PLSQR with this preconditioner and initial vector $\vec{x}_0 = \vec{0}$ are shown in the blue solid curves in \cref{fig:lsqr}; both the residual $\norm{\vec{b} - \mat{A}\smash{\vec{\hat{x}}}}$ and forward error $\norm{\vec{x} - \smash{\vec{\hat{x}}}}$ for PLSQR exceed the results for a direct solver (MATLAB's \texttt{mldivide}) by orders of magnitude.
Incidentally, we note that right-preconditioned GMRES (PGMRES) fails to meaningfully reduce the residual in 100 iterations with this preconditioner.

\begin{figure}[t]
    \centering
    \includegraphics[width=0.48\linewidth]{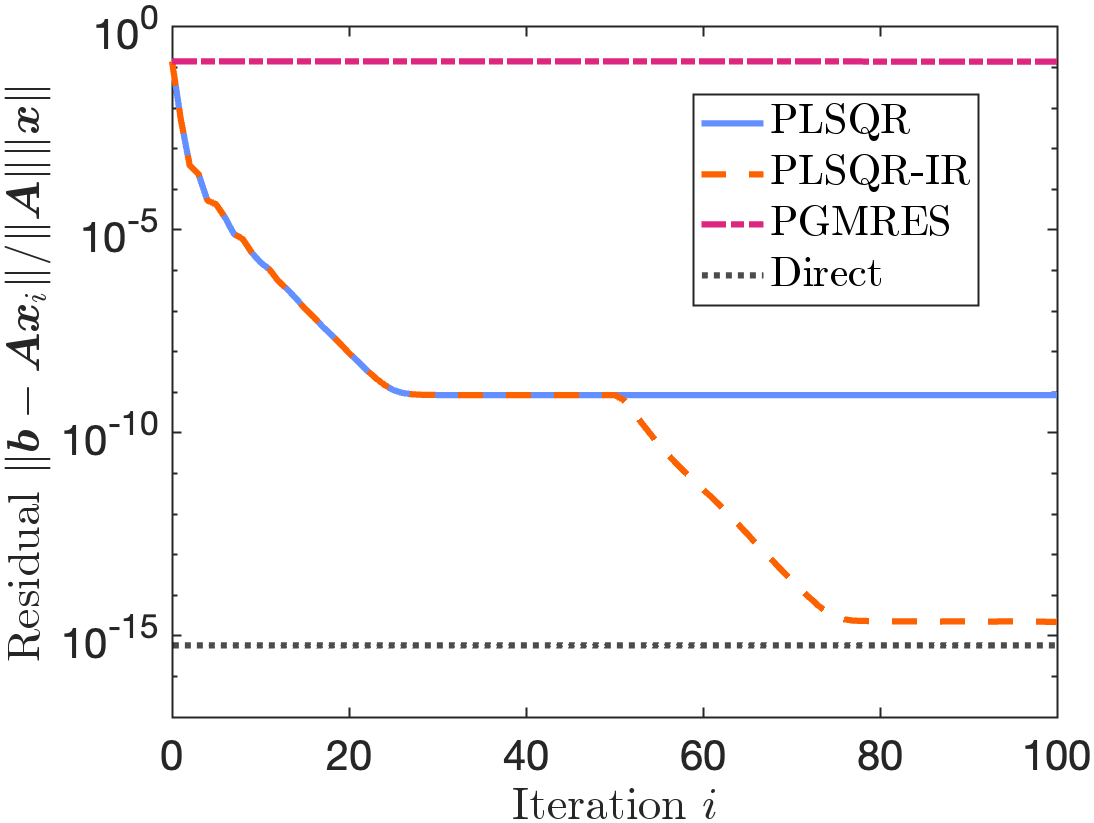}
    \includegraphics[width=0.48\linewidth]{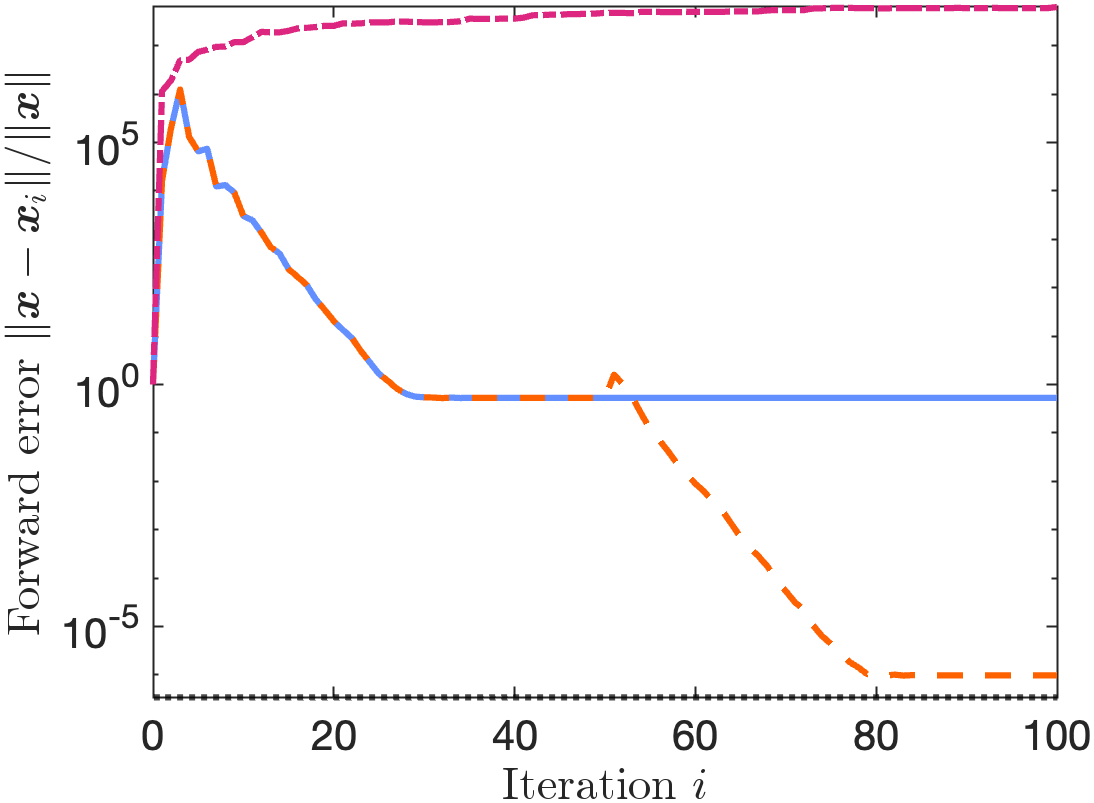}
    \caption{Residual (\emph{left}) and forward error (\emph{right}) for PLSQR and PGMRES with near-perfect preconditioner $\cond(\mat{A}\mat{P}^{-1}) = 4$.
    PGMRES fails to converge at all, and the error metrics for PLSQR exceed the direct method by orders of magnitude.
    The instabilities of PLSQR are fixed by a single step of iterative refinement (PLSQR-IR).}
    \label{fig:lsqr}
\end{figure}

This paper advocates a simple solution for the problem identified in \cref{fig:lsqr}: iterative refinement.
Suppose we run PLSQR and the residual $\norm{\vec{b} - \mat{A}\vec{x}_i}$ is observed to stagnate far above the level $\norm{\mat{A}}\norm{\vec{x}_i} u$ expected for a backward stable method; here, $\norm{\cdot}$ is the vector $\ell_2$ norm or matrix spectral norm and $u$ is the unit roundoff, of size $u\approx 10^{-16}$ in double precision.
To fix this problem, we apply (uniform precision) iterative refinement with solves done by PLSQR:
\begin{equation} \label{eq:PLSQR-IR}
    \begin{split}
    \vec{x}_1 &\gets \Call{PLSQR}{\mat{A},\mat{P},\vec{b}}, \\
    \vec{x}_2 &\gets \vec{x}_1+\Call{PLSQR}{\mat{A},\mat{P},\vec{b}-\mat{A}\vec{x}_1}.
    \end{split}
\end{equation}
We abbreviate \emph{preconditioned LSQR with iterative refinement} as PLSQR-IR.
The dashed orange curves in \cref{fig:lsqr} show results for PLSQR-IR, where each PLSQR step was individually run for 50 iterations.
We see that after the the iterative refinement step starts, the forward error and residual begin converging again until reaching machine accuracy.
We emphasize that, for difficult problems, more than one iterative refinement step may be necessary; see \cref{sec:impl-exp} for details.

\subsection{How accurate is enough?\nopunct}
The goal of this paper is to produce solutions to linear systems of equations that are, in a sense, as accurate as possible in the presence of floating point errors.
This concept is made precise by the notion of backward stability: A linear system solver is backward stable if the numerically computed solution $\vec{\hat{x}}$ is the exact solution to a machine-precision small perturbation of the system:
\begin{equation} \label{eq:backward-stable}
    (\mat{A} + \mat{\Delta A})\vec{\hat{x}} = \vec{b} \quad \text{with } \norm{\mat{\Delta A}} \le c u \norm{\mat{A}}.
\end{equation}
Here, $c > 0$ is a small prefactor, typically allowed to depend polynomially on the matrix dimension $n$.
A classical result of Rigal and Gaches \cite[Thm.~7.1]{Hig02} shows that the minimum-norm $\mat{\Delta A}$ in \cref{eq:backward-stable} satisfies
\begin{equation} \label{eq:rigal-gaches}
    \frac{\norm{\mat{\Delta A}}}{\norm{\mat{A}}} = \frac{\norm{\vec{b} - \mat{A}\vec{\hat{x}}}}{\norm{\mat{A}}\norm{\vec{\hat{x}}}}.    
\end{equation}
Thus, backward stability is equivalent to having a small residual $\norm{\vec{b}-\mat{A}\vec{\hat{x}}}$, rescaled by the norms of $\mat{A}$ and $\vec{\hat{x}}$.

In this paper, we show that, by combining PLSQR with iterative refinement, we can achieve any level of accuracy desired up to the level of backward stability.
For many applications, far less accuracy than backward stability is needed, and one can terminate when an appropriate error tolerance has been met.
We emphasize that, without iterative refinement, the errors achieved by PLSQR can be too large to be useful on some problems.
Indeed, the (relative) forward error of PLSQR saturates at 0.53 in \cref{fig:lsqr}, and just one step of iterative refinement improves the forward error by five orders of magnitude.
Therefore, the iterative refinement approach may be useful even in settings where a solution of modest accuracy is required.

\subsection{Our main theorem}

The gap between theory and practice for Krylov methods remains a chasm, and our theoretical understanding of preconditioned Krylov methods is even worse.
Despite the limited theoretical tools, we are able to \emph{prove} backward stability for a variant of the PLSQR-IR scheme.

Let us first set the stage.
LSQR treats the linear system of equations $\mat{A}\vec{x} = \vec{b}$ as a least-squares problem $\vec{x} = \argmin_{\vec{z} \in \real^n} \norm{\vec{b} - \mat{A}\vec{z}}$.
To precondition this system, we make a change of variables $\vec{y} = \mat{P}\vec{x}$: 
\begin{equation} \label{eq:pre-ls}
    \vec{y} = \argmin_{\vec{z}\in\real^n} \norm{\vec{b} - (\mat{A}\mat{P}^{-1})\vec{z}}, \quad \vec{x} = \mat{P}^{-1}\vec{y}.
\end{equation}
Now, we can convert the preconditioned system back to a linear system by forming the \emph{preconditioned normal equations}
\begin{equation} \label{eq:pre-normal}
    (\mat{P}^{-\top}\mat{A}^\top\mat{A}\mat{P}^{-1})\vec{y} = \mat{P}^{-\top}\mat{A}^\top\vec{b}.
\end{equation}
For good reason, the normal equations are viewed with suspicion by numerical analysts \cite{Hig22}, but their invocation here is safe because the preconditioned matrix $\mat{A}\mat{P}^{-1}$ has been assumed to be well-conditioned \cref{eq:good-prec}. Nonetheless, since every application of $\mat{P}^{-1}$ incurs significant numerical errors, proving backward stability is by no means straightforward; in fact, the algorithm is not backward stable without iterative refinement.

With this preparation, we can consider a variant of PLSQR-IR where we use any iterative method to solve the preconditioned normal equations \cref{eq:pre-normal}:
\actionbox{\textbf{Stable preconditioned linear system meta-solver:} Beginning from the trivial initial guess $\vec{x}_0 = \vec{0}$, do the following for $i=0,1,\ldots,t-1$:

\begin{enumerate}[label=(\alph*)]
    \item \textbf{Form the right-hand side:} Set $\vec{c}_i \gets \mat{P}^{-\top}(\mat{A}^\top(\vec{b} - \mat{A}\vec{x}_i))$.
    \item \textbf{Solve the preconditioned normal equations:} Using an iterative method, solve the system
    \begin{equation} \label{eq:pre-normal-b}
        (\mat{P}^{-\top}\mat{A}^\top\mat{A}\mat{P}^{-1}) \, \vec{\delta y}_i = \vec{c}_i.
    \end{equation}
    \item \textbf{Iterative refinement:} Update $\vec{x}_{i+1} \gets \vec{x}_i + \mat{P}^{-1}\vec{\delta y}_i$.
\end{enumerate}}
\noindent The PLSQR-IR method \cref{eq:PLSQR-IR} is equivalent, in exact arithmetic, to this meta-solver with CG or Lanczos used to solve \cref{eq:pre-normal-b}.
We emphasize that, when solving \cref{eq:pre-normal-b}, we do not explicitly form the matrix $\mat{P}^{-\top}\mat{A}^\top\mat{A}\mat{P}^{-1}$, instead applying its action on a vector $\vec{z}$ as $\mat{P}^{-\top}(\mat{A}^\top(\mat{A}(\mat{P}^{-1}\vec{z})))$.

For our main theoretical results, we will analyze this meta-solver using the Lanczos method to solve \cref{eq:pre-normal-b}; see \cref{sec:analysis} for the specific Lanczos implementation we are using.
Our main theorem may be summarized as follows:

\begin{inftheorem}[Preconditioning the normal equations with iterative refinement] \label{infthm:main}
    Let $\mat{A}$ be a numerically full-rank matrix, for which $\cond(\mat{A})u \ll 1$, and suppose $\mat{P}$ is a near-perfect preconditioner \cref{eq:good-prec}.
    Assume that the operations $\mat{A}\vec{z}$, $\mat{A}^\top\vec{z}$, $\mat{P}^{-1}\vec{z}$, $\mat{P}^{-\top}\vec{z}$ are performed in a backward stable way, and use the Lanczos method (\cref{alg:lanczos}) to solve the system \cref{eq:pre-normal-b}.
    With $t \le \order(\log(1/u) / \log(1/(\kappa u)))$ iterative refinement steps
    and $q \le \order(\log(1/(\kappa u)))$ iterations per Lanczos solve, the meta-algorithm outputs a backward stable solution to $\mat{A}\vec{x} = \vec{b}$.
    The total number of Lanczos steps across all iterative refinement steps is at most $\order(\log(1/u))$, and the primitives $\mat{A}\vec{z}$, $\mat{A}^\top \vec{z}$, $\mat{P}^{-1}\vec{z}$, and $\mat{P}^{-\top}\vec{z}$ are each invoked at most $\order(\log(1/u))$ times.
\end{inftheorem}

See \cref{thm:main-lanczos} for a formal statement and \cref{thm:main-meta} for a generalization.
The choice of Lanczos as the solver for \cref{eq:pre-normal-b} may seem odd, but we remind the reader that LSQR, CG on the normal equations, and Lanczos on the normal equations are all equivalent in exact arithmetic.
Of the three algorithms, we view Lanczos as having the most mature analysis in finite precision and LSQR as having the least.
The general version of \cref{infthm:main} treats a general solver for \cref{eq:pre-normal-b}, allowing us to obtain results for PLSQR-IR should better analysis of LSQR be developed in the coming years.
Given our numerical testing and \cref{infthm:main}, we regard the backward stability of PLSQR-IR to be almost certain.

\subsection{Related work}
We highlight connections of our work to three areas of related work: restarting Krylov methods, Krylov methods for iterative refinement, and backward stable randomized least-squares algorithms.

\paragraph{Restarting Krylov methods}
Our proposal is equivalent, in exact arithmetic, to restarting the PLSQR method, i.e.,
\begin{equation} \label{eq:LSQR-restarting}
    \begin{split}
    \vec{x}_1 &\gets \Call{PLSQR}{\mat{A},\mat{P},\vec{b},\texttt{initial}=\vec{0}}, \\
    \vec{x}_2 &\gets \Call{PLSQR}{\mat{A},\mat{P},\vec{b},\texttt{initial}=\vec{x}_1}.
    \end{split}
\end{equation}
(Compare \cref{eq:PLSQR-IR}.)
As best we can tell, the proposal to use restarting or iterative refinement \emph{to improve the attainable accuracy of preconditioned Krylov methods} is somewhat novel.
Some references (e.g., \cite[p.~49]{She94a}) advocate for periodically recomputing the residual in CG using the formula $\vec{b} - \mat{A}\vec{x}_i$ to improve its resilience to finite precision effects, which has a similar effect to restarting; however, this technique is an empirical fix and no stability analysis is available.
Residual recomputing does not appear to be a widely adopted practice, and MATLAB's \texttt{lsqr} and \texttt{pcg} methods terminate with errors many orders of magnitude larger than direct methods for the examples considered in this paper.
It is not obvious whether residual recomputing is applicable at all with LSQR.

In contrast to CG and LSQR, restarting the GMRES method is common.
We emphasize that the motivations for restarting/iterative refinement for PLSQR in our work and the typical motivations for restarting GMRES are entirely separate: Restarting for GMRES is done periodically to control computational costs associated with storing and orthogonalizing the full Krylov subspace; restarting GMRES can \emph{hurt} the accuracy and convergence of the method.
By contrast, our work restarts PLSQR only at judicious times, which leads to significant \emph{improvements} in the numerical accuracy. 

\paragraph{Krylov methods as an alternative for iterative refinement}
An alternative line of work (see, e.g., \cite{amestoy2024five,carson2018accelerating,HP21,CH17}) has used preconditioned Krylov methods like GMRES or CG as \emph{replacements} for iterative refinement for solving linear systems of equations using multiple precisions.
The basic idea is this:
Given access to a low-precision \LU factorization $\mat{A} \approx \mat{\hat{L}}\mat{\hat{U}} \eqqcolon \mat{P}$ computed in low precision $u' \gg u$, we may solve $\mat{A}\vec{x} = \vec{b}$ using GMRES with preconditioner $\mat{P}$, applied via triangular substitution.
While this work shares some similarity to ours, their aims are fundamentally different: They use preconditioned Krylov methods to perform the role of iterative refinement in at least two precisions $u', u$, whereas we use iterative refinement to improve the attainable accuracy of preconditioned Krylov methods using iterative refinement in a uniform precision $u$.

\paragraph{Backward stable randomized least squares}
Our paper builds on recent work \cite{EMN24} by a subset of the present authors and Maike Meier, which demonstrated that, provided access to both a high-quality preconditioner \emph{and a high-quality initialization}, a class of preconditioned iterative methods for least squares achieve backward stability with a single step of iterative refinement, resolving the open question \cite{Epp24a,MNTW24} of whether \emph{any} fast randomized least-squares solver was backward stable.
The present paper differs from \cite{EMN24} in that we consider square systems, we do not assume access to a high-quality initialization, and PLSQR-IR generally requires more than one step of iterative refinement to attain backward stability.
We emphasize that our results \emph{are not} a consequence of the main theorem in \cite{EMN24}.

\subsection{Main ideas and outline}
This paper contains three main ideas:
\begin{enumerate}[label=(\roman*)]
    \item With iterative refinement and a sufficiently good preconditioner, iterative methods based on the normal equations are fast and backward stable.
    \item In many ways, the problem of a designing an ``LSQR-good'' preconditioner is easier than designing a ``GMRES-good'' preconditioner.
    For this reason, we argue that (P)LSQR may be a more attactive solver for nonsymmetric linear systems than has been previously realized.
    \item Empirically, \emph{left-preconditioned} LSQR appears to be backward stable \emph{even without iterative refinement}.
\end{enumerate}
The first idea is a theorem, the second is a perspective, and the third is an empirical observation.
We discuss (ii) in \cref{sec:lsqr}, and provide evidence for (iii) in \cref{sec:left-prec-lsqr}.
We discuss positive definite systems in \cref{sec:pd}, implementation details for PLSQR-IR appear in \cref{sec:impl-exp}., and backward error analysis appears in \cref{sec:analysis}.

\subsection{Reproducible research}
Code for the methods and experiments in this paper are available at \url{https://github.com/eepperly/Stable-Preconditioned-Solvers}.

\section{LSQR versus GMRES} \label{sec:lsqr}

What iterative method should one use to solve nonsymmetric linear systems?
For many, the natural answer would be (preconditioned) GMRES, and it has been the de facto choice in a number of application domains.
This section will critically evaluate this choice and identify situations in which methods based on the normal equations like LSQR may be more effective.

\subsection{Illustrative examples}

From the start, let us emphasize that there are examples of nonsymmetric matrices $\mat{A}$ for which GMRES dramatically outperforms LSQR (equivalently in exact arithmetic, CG on the normal equations) and visa versa, a phenomenon which is vividly demonstrated in the famous paper \cite{NRT92}.
To set the stage for our discussion and to illustrate some of the possible behaviors of GMRES and LSQR, we highlight three example matrices of size $n=1000$.
To build these matrices, we generate a well-conditioned matrix $\mat{W} \in \real^{n\times n}$ with $\cond(\mat{W}) = 2$ and Haar-random orthogonal matrix $\mat{U} \in \real^{n\times n}$, and define diagonal matrices
\begin{align*}
    \mat{D} &= \texttt{diag(logspace(-2,0,n))} = \diag\left( 10^{-2(i-1)/(n-1)} : i = 1,\ldots,n \right), \\
    \mat{D}' &= \texttt{diag(1./linspace(1,4,n))} = \diag\left( \frac{1}{1 + 3(i-1)/(n-1)} : i = 1,\ldots,n \right).
\end{align*}
We define test matrices 
\begin{align*}
    \texttt{ClusterEigsIllCond} &= \mat{W}\mat{D}\mat{W}^{-1}, \\
    \texttt{SpreadEigsWellCond} &= \mat{U}\mat{D}', \\
    \texttt{ClusterEigsWellCond} &= \mat{W}\mat{D}'\mat{W}^{-1}.
\end{align*}
As the names suggest, the first and third matrices have real eigenvalues clustered in an interval of the form $[a,1]$, and the second matrix has its eigenvalues spread out over an annulus centered at zero in the complex plane.
The first matrix is mildly ill-conditioned (condition number $\approx 100$), and the second and third matrices are well-conditioned (condition number $\approx 4$). 

\begin{figure}[t]
    \centering
    \includegraphics[width=0.99\linewidth]{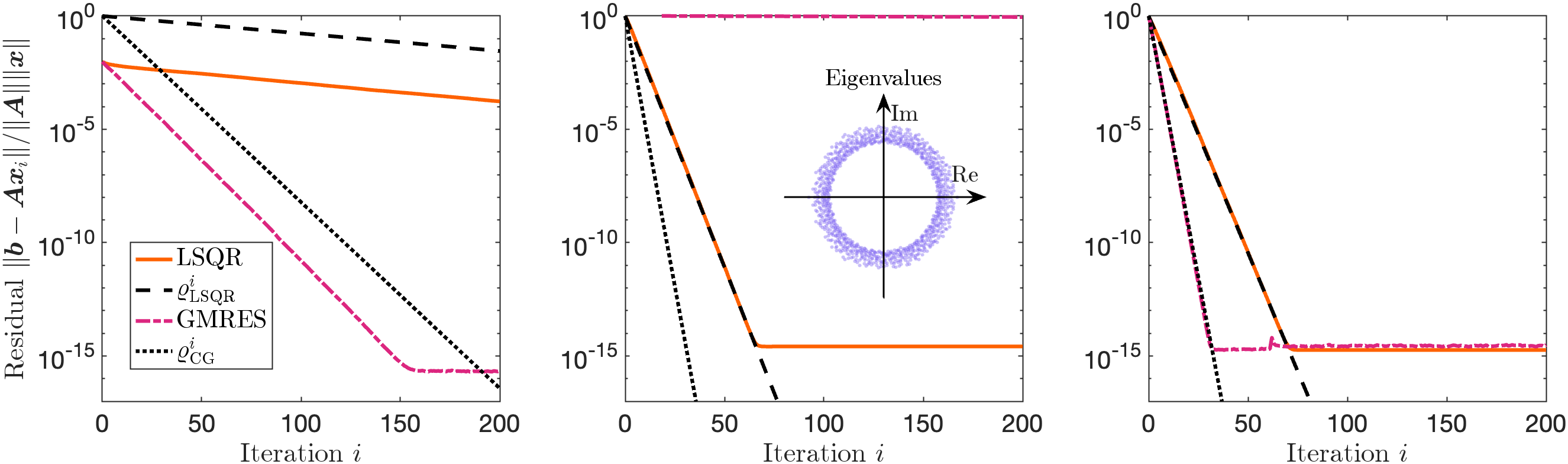}

    \parbox{0.32\linewidth}{\centering \texttt{ClusterEigsIllCond}}
    \parbox{0.32\linewidth}{\centering \texttt{SpreadEigsWellCond}}
    \parbox{0.32\linewidth}{\centering \texttt{ClusterEigsWellCond}}
    
    \caption{Convergence of LSQR and GMRES on three test matrices described in the text.
    The spectrum of the second matrix is shown as an inset on the second panel.
    Convergence rates for CG and LSQR for a matrix of condition number $\kappa$ are shown.}
    \label{fig:gmres-lsqr}
\end{figure}

Results for GMRES and LSQR are shown in \cref{fig:gmres-lsqr}.
The first panel shows the good case for GMRES: The matrix \texttt{ClusterEigsIllCond} is nearly normal and has clustered eigenvalues, and the rate of convergence for GMRES is comparable to that of CG when applied to a symmetric positive definite system with the same condition number, $\varrho_{\rm CG} \coloneqq (\sqrt{\kappa} - 1)/(\sqrt{\kappa}+1)$ for $\kappa = \cond(\mat{A})$.
As LSQR is equivalent to CG applied to the normal equations, its convergence rate is much slower, $\varrho_{\rm LSQR} \coloneqq (\kappa-1)/(\kappa+1)$.
The second panel shows the good case for LSQR: The matrix \texttt{SpreadEigsWellCond} is well-conditioned, but its eigenvalues are spread out in a ring on the complex plane, causing a complete failure in GMRES.
The final panel shows an example where the eigenvalues are clustered and the matrix is well-conditioned, and both LSQR and GMRES  converge at fast rates.

\subsection{Convergence theories}

The differing behaviors of (preconditioned) GMRES and LSQR in \cref{fig:lsqr,fig:gmres-lsqr} is well-described by existing theory.
As we have mentioned, preconditioned LSQR is equivalent to conjugate gradient on the preconditioned normal equations \cref{eq:pre-normal}.
In exact arithmetic, we have the standard bound
\begin{equation} \label{eq:lsqr-conv}
    \norm{\vec{b} - \mat{A}\vec{x}_k} \le 2\varrho_{\rm LSQR}^k\norm{\vec{b}} \quad \text{with } \varrho_{\rm LSQR} = \frac{\cond(\mat{A}\mat{P}^{-1}) - 1}{\cond(\mat{A}\mat{P}^{-1}) + 1}.
\end{equation}
In particular, we obtain rapid convergence under the sole condition that $\mat{A}\mat{P}^{-1}$ is well-conditioned \cref{eq:good-prec}.

The convergence theory for (preconditioned) GMRES is more nuanced.
The classical result shows that the residual after $k$ steps of PGMRES satisfies the bound
\begin{equation} \label{eq:gmres-conv}
    \norm{\vec{b} - \mat{A}\vec{x}_k} \le \kappa_{\rm V}(\mat{A}\mat{P}^{-1}) \cdot \min_{\substack{\deg p \le k \\ p(0) = 1}} \max_{\lambda \in \Lambda(\mat{A}\mat{P}^{-1})} |p(\lambda)| \cdot \norm{\vec{b}}.
\end{equation}
Here, $\kappa_{\rm V}(\mat{B})$ is the minimal condition number of any eigenvector matrix of $\mat{B}$, the minimum is taken over all degree-$k$ polynomials with value $1$ at the origin, and the maximum is taken over the spectrum of $\mat{A}\mat{P}^{-1}$.
This result suggests that to obtain a high-quality preconditioner for GMRES requires ensuring that the preconditioned matrix $\mat{A}\mat{P}^{-1}$ is nearly normal (so that $\kappa_{\rm V}(\mat{A}\mat{P}^{-1})$ is small) and has a spectrum over which there exists a good polynomial approximation to the function $1/z$.
For this latter condition to hold, the eigenvalues should be clustered on a geometrically compact region that is well-separated from and does not wrap around the origin.

A famous example shows that, even for an orthogonal matrix $\mat{A}$ (which is perfectly conditioned $\cond(\mat{A}) = 1$), GMRES can take $n$ steps to produce a nontrivial result:
\begin{equation} \label{eq:gmres-bad}
    \mat{A} = \twobytwo{\vec{0}_{1\times (n-1)}}{1}{\Id_{n-1}}{\vec{0}_{(n-1)\times 1}}, \vec{b} = \twobyone{1}{\vec{0}_{n-1}}.
\end{equation}
The eigenvalues of this matrix form a set of $n$th roots of unity, which are perfectly ``anti-clustered''.
The middle panel of \cref{fig:gmres-lsqr} provides a similar example.

\subsection{Why preconditioned LSQR?\nopunct} \label{sec:why-preconditioned-lsqr}

As both theory and the examples in \cref{fig:gmres-lsqr} show, either GMRES or LSQR could be the appropriate method for a particular matrix $\mat{A}$.
But with preconditioning, we have a mechanism to \emph{engineer} the matrix $\mat{A}\mat{P}^{-1}$ to have desirable properties.
This raises the question: \emph{Is it easier to design a good preconditioner for LSQR or GMRES?}
While we do not claim to definitively answer this question, we believe there are reasons to expect that finding an \emph{LSQR-good} preconditioner should be easier than finding a \emph{GMRES-good} preconditioner.

Our main argument can be summarized as follows:
\actionbox{An LSQR-good preconditioner must only be a \emph{near-orthogonalizer}, whereas a GMRES-good preconditioner must be a \emph{near-inverse}.}
\noindent This statement is supported by both theory and experiment.
First, observe that empirically in \cref{fig:gmres-lsqr} and theoretically in \cref{eq:lsqr-conv}, the convergence rate of LSQR is governed by the condition number $\cond(\mat{A}\mat{P}^{-1})$, which measures the \emph{near-orthogonality} and \emph{singular value clustering} of $\mat{A}\mat{P}^{-1}$.
In particular, $\cond(\mat{A}\mat{P}^{-1}) = 1$ if and only if $\mat{A}\mat{P}^{-1}$ is a scalar multiple of an orthogonal matrix (equivalently, all singular values are clustered at a single point).
For GMRES, the middle panel of \cref{fig:gmres-lsqr} and the example \cref{eq:gmres-bad} demonstrate that the performance of GMRES can be arbitrarily poor if $\mat{P}$ is just a near-orthogonalizer.
Instead, as demonstrated in \cref{eq:gmres-conv} and the left and right panels of \cref{fig:gmres-lsqr}, a GMRES-good preconditioner $\mat{P}$ must be a near-inverse in the sense that it clusters the eigenvalues of $\mat{A}\mat{P}^{-1}$ to a geometrically compact region.
(And, in principle at least, a GMRES-good preconditioner must make $\mat{A}\mat{P}^{-1}$ have well-conditioned eigenvectors, in view of \cref{eq:gmres-conv}.)

It is natural to suspect that finding a near-orthogonalizing preconditioner (i.e., an LSQR-good preconditioner) should be easier than finding a near-inverse preconditioner. 
After all, a near-inverse preconditioner seeks to bring $\mat{A}\mat{P}^{-1}$ towards the single matrix $\Id$ (or more precisely, a scaling $\alpha \Id$), where a near-orthogonalizing preconditioner must only bring $\mat{A}\mat{P}^{-1}$ close to the set of orthogonal matrices, which is a manifold of dimension $\approx n^2/2$.
The much larger dimensionality of near-orthogonal matrices suggests a wide design space for designing preconditioners for LSQR, including ideas such as low-precision or partial \QR factorizations \cite{LS06,CD25}, approximate Cholesky factors of the Gram matrix $\mat{A}^\top\mat{A}$, or the ``LV'' part of a \ULV decomposition of a rank-structured matrix \cite{CGP06}.
(So far, much of the work on these types of preconditioners has focused on overdetermined least-squares problems, but the ideas should be transferable to square linear systems as well.)
The work \cite{LLYNUZ2025}, which is being released concurrently with ours, also explores solving nonsymmetric linear systems by applying preconditioned iterative methods to the normal equations. That paper adopts a perspective based on functional analysis and discretization of PDEs, and it includes numerical evaluations of several preconditioners for use with LSQR/CG on the normal equations.

The GMRES method certainly has advantages over LSQR, which we review in \cref{sec:gmres-strengths}, so we do not want our case for preconditioned LSQR to be taken as too absolute.
Still, we argue that the design of good preconditioners for nonsymmetric linear systems for use with LSQR is an underexplored area.
Given the position of GMRES as the most popular Krylov method for nonsymmetric linear systems, significant algorithm design effort has gone into the design of preconditioners, adapted to GMRES's strengths and attempting to compensate for its weaknesses.
A small fraction of that effort has been paid to developing near-orthogonalizing, LSQR-good preconditioners.
We are optimistic that many exciting preconditioning strategies lie in waiting to be discovered for use in solving linear systems with normal equation-based solvers like LSQR.

\section{The mysterious stability of left-preconditioned LSQR} \label{sec:left-prec-lsqr}

So far, we have focused on \emph{right-preconditioned} LSQR, which we observe to be backward stable when combined with iterative refinement.
Our main theoretical results (\cref{infthm:main}) support these observations with rigorous proof, up to replacing PLSQR with the Lanczos method applied to the preconditioned normal equations.

\begin{figure}
    \centering
    \includegraphics[width=0.48\linewidth]{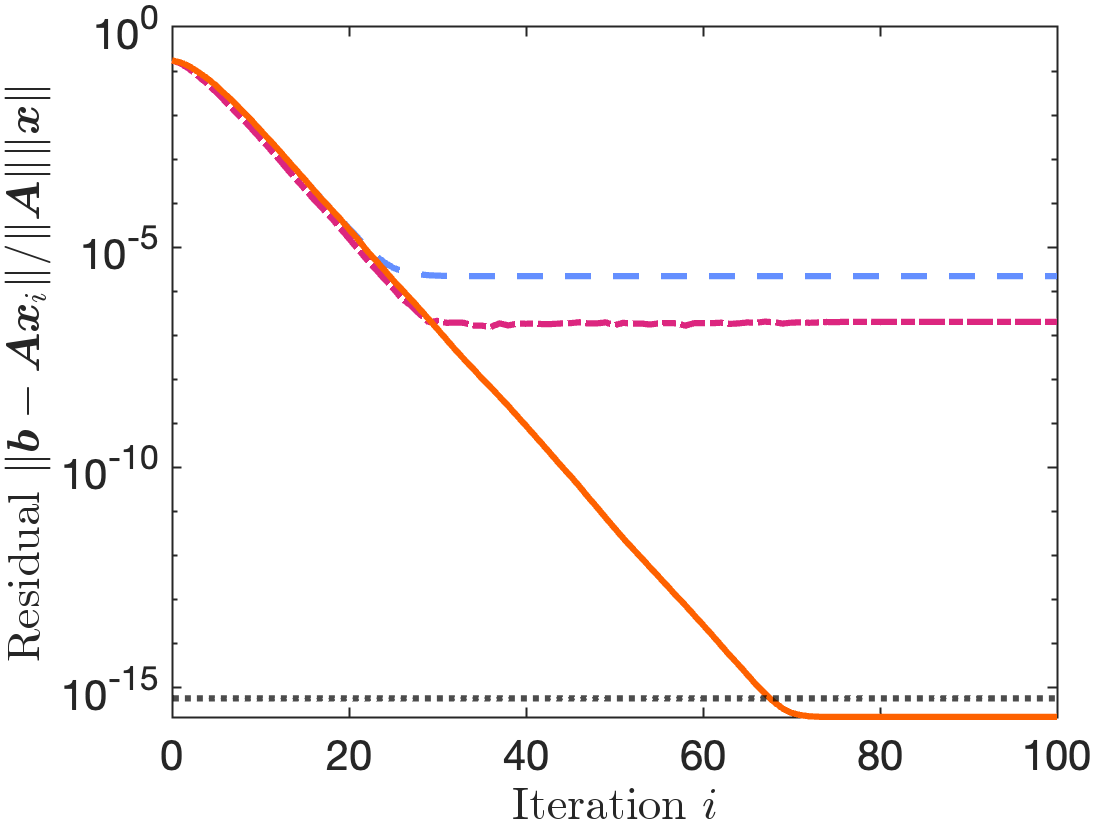}
    \includegraphics[width=0.48\linewidth]{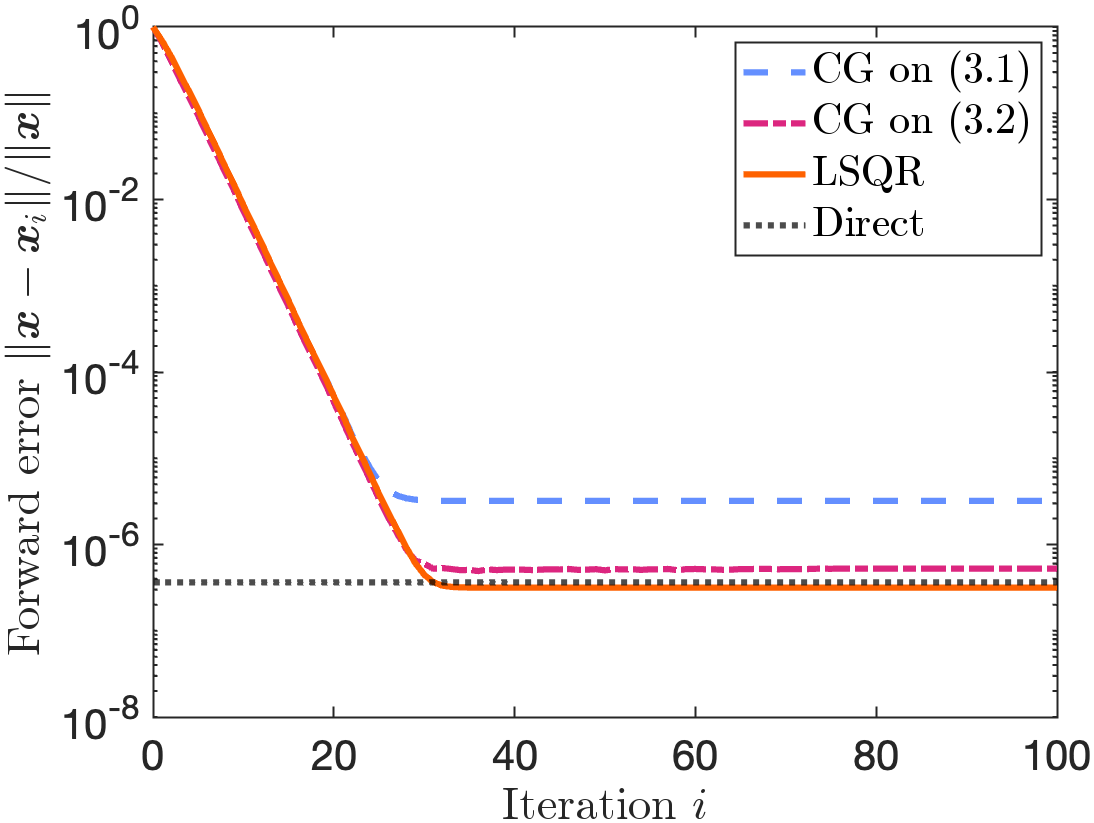}
    
    \caption{Residual (\emph{left}) and forward error (\emph{right}) for left-preconditioned LSQR and CG on the normal equations \cref{eq:left-pre-normal} and adjoint normal equations \cref{eq:left-pre-adj-normal}.}
    \label{fig:left-lsqr}
\end{figure}

During numerical testing, we uncovered evidence that \emph{left-preconditioned LSQR} satisfies even stronger stability properties, achieving backward stability \emph{without iterative refinement}.
This phenomenon is illustrated in \cref{fig:left-lsqr}.
Here, we test with a random matrix $\mat{A}$ with $\cond(\mat{A}) = 10^{10}$ and a near-perfect \emph{left} preconditioner $\mat{P}$ satisfying $\cond(\mat{P}^{-1}\mat{A}) = 4$; see \cref{sec:experiment-details} for details.
The orange solid line shows LSQR on the left-preconditioned linear system
\begin{equation*}
    \mat{P}^{-1}\mat{A}\vec{x} = \mat{P}^{-1}\vec{b},
\end{equation*}
which converges geometrically until reaching machine accuracy.
We emphasize that LSQR is executed \emph{without restarting or iterative refinement} in this experiment.

The strong stability properties exhibited by left-preconditioned LSQR in this experiment are especially surprising because they do not manifest when CG is used in place of LSQR.
Indeed, the blue dashed and magenta dash-dotted curves show conjugate gradient applied to the left-preconditioned normal equations
\begin{equation} \label{eq:left-pre-normal}
    (\mat{A}^\top\mat{P}^{-\top}\mat{P}^{-1}\mat{A})\vec{x} = \mat{A}^\top(\mat{P}^{-\top}(\mat{P}^{-1}\vec{b}))
\end{equation}
or the left-preconditioned adjoint normal equations
\begin{equation} \label{eq:left-pre-adj-normal}
    (\mat{P}^{-1}\mat{A}\mat{A}^\top\mat{P}^{-\top})\vec{y} = \mat{P}^{-1}\vec{b} \quad \text{with } \vec{x} = \mat{A}^\top(\mat{P}^{-\top}\vec{y}).
\end{equation}
In both cases, the backward error saturates ten orders of magnitude above machine accuracy.
(The forward error, however, is comparable to the direct solve and left-preconditioned LSQR.)

This experiment illustrates a fundamental disjunction between the stability properties of left and right preconditioning:
For right preconditioning, iterative refinement is needed to obtain backward stability and the numerical behavior of the method is robust to the precise solver used to solve the preconditioned normal equations (see \cref{thm:main-meta}).
For left preconditioning, backward stability is obtained for left-preconditioned LSQR without iterative refinement and these stability properties do not carry over to other methods for solving the left-preconditioned normal equations like CG.
At present, we do not have a theoretical explanation for the observed backward stability of left-preconditioned LSQR.
We have done additional numerical testing beyond what is reported here and have always observed left-preconditioned LSQR to be backward stable when the matrix $\mat{A}$ is numerically full-rank $\cond(\mat{A})\le 0.1/u$ and the preconditioner is good $\cond(\mat{P}^{-1}\mat{A}) \le 10$.

\section{Symmetric positive definite systems} \label{sec:pd}

\begin{figure}
    \centering
    \includegraphics[width=0.9\linewidth]{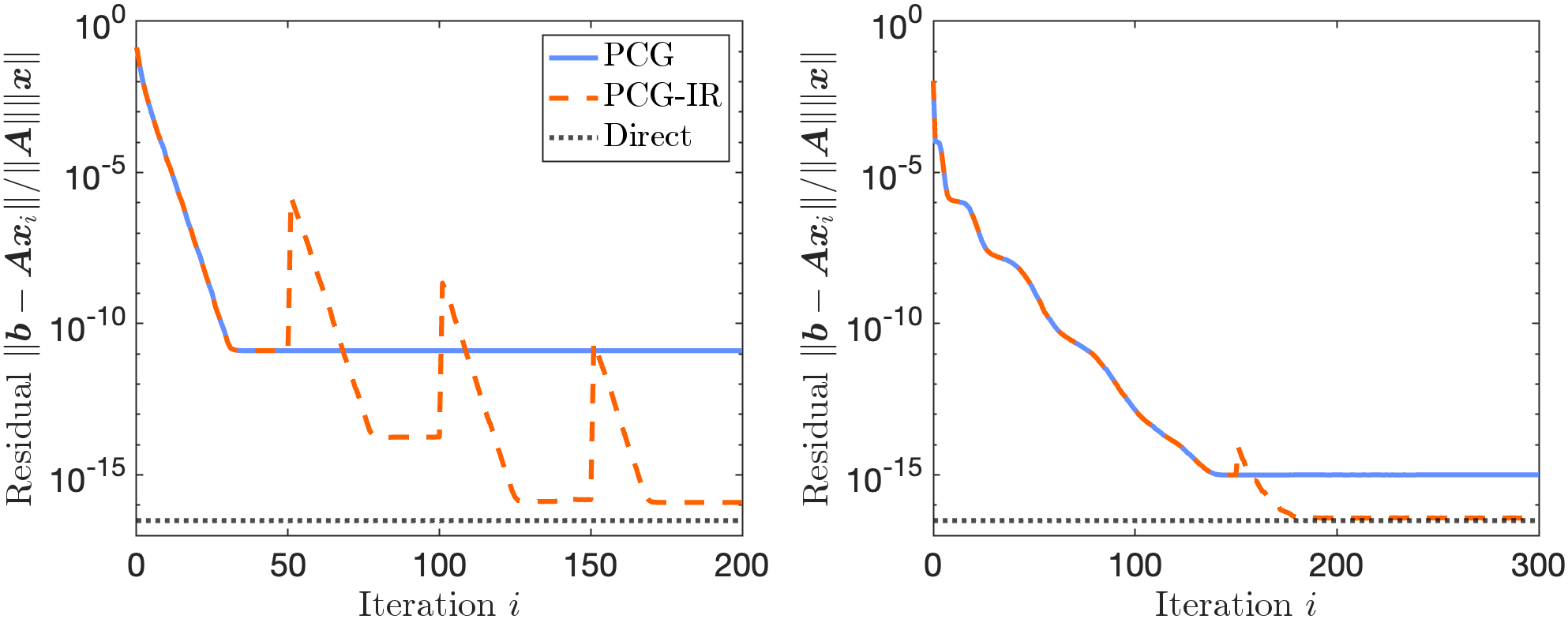}
    \caption{Residual for PCG applied to two example matrices, described in the text.
    PCG with iterative refinement every 50 (\emph{left}) or 150 (\emph{right}) iterations is seen to improve the attainable accuracy by several orders of magnitude.}
    \label{fig:pd}
\end{figure}

While this paper is largely concerned with nonsymmetric systems of linear equations, we mention that iterative refinement can also help improve the maximum achievable accuracy for preconditioned Krylov methods for symmetric positive definite (spd) systems of linear equations.
An illustration is provided in \cref{fig:pd}.
The left panel shows a synthetic matrix with 
size $n=10^3$, condition number $\cond(\mat{A}) = 10^{10}$, and near-perfect preconditioner $\mat{P}$ satisfying $\cond(\mat{P}^{-1/2}\mat{A}\mat{P}^{-1/2}) \approx 1.3$.
The second example is obtained by applying kernel ridge regression to the \texttt{COMET\_MC\_SAMPLE} data set, resulting in linear system of dimension $n = 10^4$ with condition number $\cond(\mat{A})\approx 6\times 10^{13}$.
We employ a Nystr\"om preconditioner of rank $k = 850$ computed using the randomly pivoted Cholesky algorithm \cite{CETW25,DEF+23}; the preconditioner is of modest quality, $\cond(\mat{P}^{-1/2}\mat{A}\mat{P}^{-1/2}) \approx 200$.
See \cref{sec:experiment-details} for more information.

In both examples, preconditioned conjugate gradient without iterative refinement (blue solid line) is not backward stable.
The orange dashed lines show PCG with iterative refinement at every 50 (\emph{left}) or 150 (\emph{right}) iterations; backward stability is attained after two or one steps of iterative refinement respectively.
We observe that the residual jumps for each method after iterative refinement.
Theoretical analysis and more extensive numerical testing of iterative refinement for solving spd linear systems is a natural subject for future work.

\section{LSQR-IR implementation} \label{sec:impl-exp}

To achieve the good performance in practice, we propose an implementation of (P)LSQR-IR that automates the decisions of when to restart and when to terminate.
We begin by estimating the spectral norm $\norm{\mat{A}}$ by applying $\lceil \log n\rceil$ steps of the randomized power method \cite[sec.~6.2]{MT20b}, resulting in $\mathrm{normest} \approx \norm{\mat{A}}$. (All logarithms in this paper are natural.)
Every $f$ iterations, where $f$ is a user-specified parameter, we compute the backward error estimate
\begin{equation*}
    \mathrm{berr}_i \coloneqq \frac{\norm{\vec{b} - \mat{A}\vec{x}_i}}{\mathrm{normest} \cdot \norm{\vec{x}_i}}.
\end{equation*}
(Recall \cref{eq:rigal-gaches}.)
If $\mathrm{berr}_i$ is less than a tolerance (we use $n^{1/2}u$), we terminate.
Otherwise, we check whether the backward error has stagnated: If $\mathrm{berr}_i > 0.9 \cdot \mathrm{berr}_{i-f}$, we stop the current PLSQR run and begin iterative refinement.

We make a few other comments on implementation:
\begin{itemize}
    \item \textbf{Restarting vs.\ iterative refinement.} In our experience, iterative refinement \cref{eq:PLSQR-IR} and restarting \cref{eq:LSQR-restarting} have similar convergence rates and stability properties.
    We show results only for iterative refinement in this paper.
    \item \textbf{``Preconditioned LSQR.''}
    When preconditioning LSQR, it is usual to run LSQR directly on the preconditioned problem \cref{eq:pre-ls}, solving for $\vec{y} = \argmin_{\vec{z}} \norm{\vec{b} - (\mat{A}\mat{P}^{-1})\vec{z}}$, then applying the inverse-preconditioner to obtain $\vec{x} = \mat{P}^{-1}\vec{y}$.
    We use this approach in our work.
    Alternately, there are purpose-built preconditioned LSQR implementations that track approximate solutions $\vec{x}_i\approx \vec{x}$ rather than its transformation $\vec{y}_i = \mat{P}\vec{x}_i$ \cite[sec.~4.5]{Mei24}.
\end{itemize}
Pseudocode for this PLSQR-IR implementation is provided in \cref{alg:lsqr-ir}.
Our code can be found at \url{https://github.com/eepperly/Stable-Preconditioned-Solvers}.

\begin{algorithm}[t]
	\caption{Automatic implementation of PLSQR-IR} \label{alg:lsqr-ir}
	\begin{algorithmic}[1]
		\Require Subroutines \Call{Apply}{}, \Call{ApplyT}{}, \Call{Pre}{}, \Call{PreT}{} implementing $\vec{z} \mapsto \mat{A}\vec{z},\mat{A}^\top\vec{z},\mat{P}^{-1}\vec{z},\mat{P}^{-\top}\vec{z}$, right-hand side $\vec{b} \in \real^n$, check frequency $f$
		\Ensure Backward stable solution $\vec{x} \approx \mat{A}^{-1}\vec{b}$
		\State $\vec{z} \gets \Call{GaussianVector}{n}$ \Comment{Random initialization}
        \For{$i=1,\ldots,\lceil\log(n)\rceil$}
        \State $\vec{z} \gets \Call{ApplyT}{\textsc{Apply}(\vec{z})}$, $\vec{z} \gets \vec{z} / \norm{\vec{z}}$ \Comment{Power iteration}
        \EndFor
        \State $\mathrm{normest} \gets \norm{\Call{ApplyT}{\vec{z}}}$ \Comment{Estimate for $\norm{\mat{A}}$}
        \State $\beta \gets \norm{\vec{b}}$, $\vec{u} \gets \vec{b} / \beta$, $\vec{y} \gets \vec{0}$, $\vec{x}\gets \vec{0}$, $\mathrm{berr}_0 \gets \infty$ \Comment{Initialize for first solve}
        \While{\texttt{true}} \Comment{Keep going until backward stable solution is found}
        \State $\vec{v} \gets \Call{PreT}{\textsc{ApplyT}(\vec{u})}$, $\alpha \gets \norm{\vec{v}}$, $\vec{v} \gets \vec{v} / \alpha$, $\vec{w} \gets \vec{v}$ \Comment{Initialize each solve}
        \State $\overline{\phi} \gets \beta$, $\overline{\varrho} \gets \alpha$
        \For{$i=1,2,\ldots$}
        \State $\vec{u}\gets \Call{Apply}{\textsc{Pre}(\vec{v})} - \alpha \vec{u}$, $\beta \gets \norm{\vec{u}}$, $\vec{u} \gets \vec{u}/\beta$ \Comment{Apply $\mat{A}^{-1}\mat{P}^{-1}$}
        \State $\vec{v} \gets \Call{PreT}{\textsc{ApplyT}(\vec{u})}-\beta \vec{v}$, $\alpha \gets \norm{\vec{v}}$, $\vec{v} \gets \vec{v} / \alpha$ \Comment{Apply $\mat{P}^{-\top}\mat{A}^\top$}
        \State $\varrho \gets \sqrt{\overline{\varrho}^2 + \beta^2}$
        \State $c \gets \overline{\varrho} / \varrho$, $s \gets \beta / \varrho$
        \State $\theta \gets s \alpha$, $\overline{\varrho} \gets -c\alpha$, $\phi\gets c\overline{\phi}$, $\overline{\phi}\gets s\overline{\phi}$
        \State $\vec{y} \gets \vec{y} - \phi/\varrho\cdot \vec{w}$, $\vec{w} \gets \vec{v} - \theta/\varrho\cdot \vec{w}$ \Comment{Update solution $\vec{y}$, search direction}
        \If{$\Call{Mod}{i,f} = 0$} \Comment{Check termination/restarting every $f$ iterations}
        \State $\vec{x}' \gets \vec{x}+\Call{Pre}{\vec{y}}$, $\vec{r} \gets \vec{b} - \Call{Apply}{\vec{x}'}$ \Comment{Compute solution $\vec{x}$, residual}
        \State $\mathrm{berr}_i = \norm{\vec{r}} / (\norm{\vec{x}'} \cdot \mathrm{normest})$ \Comment{Backward error estimate}
        \State \textbf{if} $\mathrm{berr}_i \le n^{1/2}u$ \textbf{then return} $\vec{x}'$ \Comment{Terminate if backward stable}
        \State \textbf{if} $\mathrm{berr}_i > 0.9\cdot \mathrm{berr}_{i-f}$ \textbf{then break} \Comment{Restart if stagnated}
        \EndIf
        \EndFor
        \State $\vec{x} \gets \vec{x}'$, $\vec{y} \gets \vec{0}$, $\vec{u} \gets \vec{r}$, $\beta \gets \norm{\vec{u}}$, $\vec{u} \gets \vec{u} /\beta$ \Comment{Iterative refinement}
        \EndWhile
	\end{algorithmic}
\end{algorithm}

\begin{figure}
    \centering
    \includegraphics[width=0.48\linewidth]{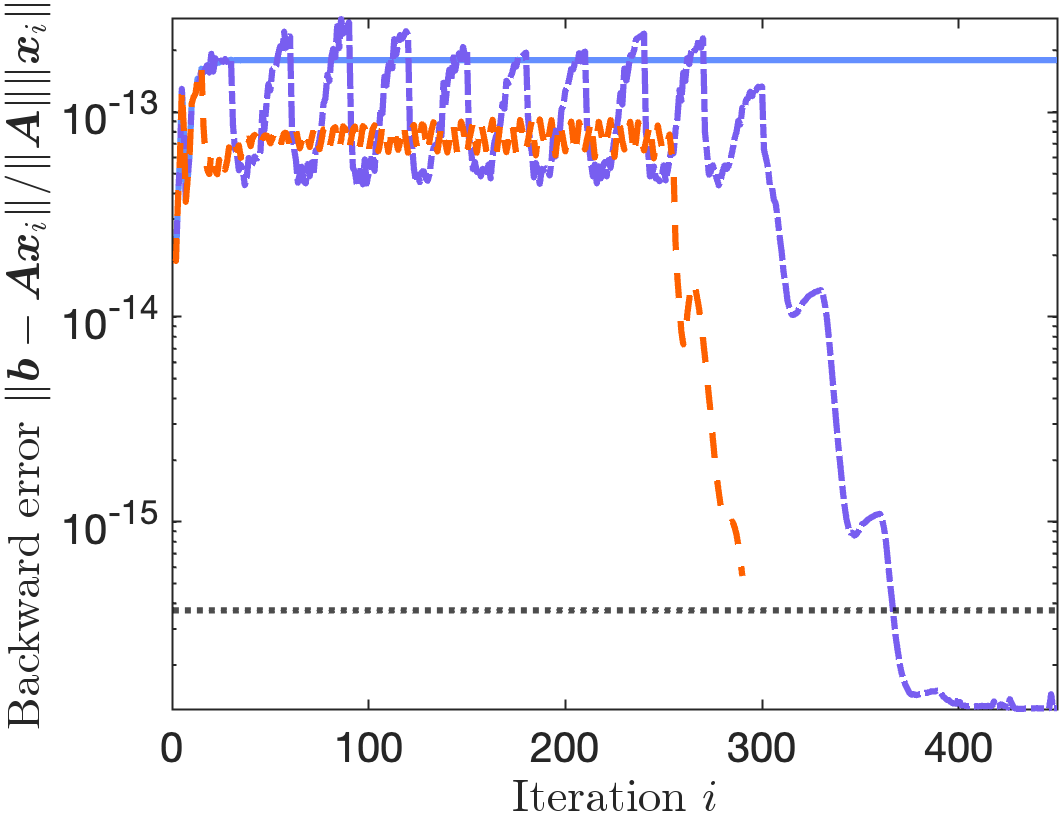}
    \includegraphics[width=0.48\linewidth]{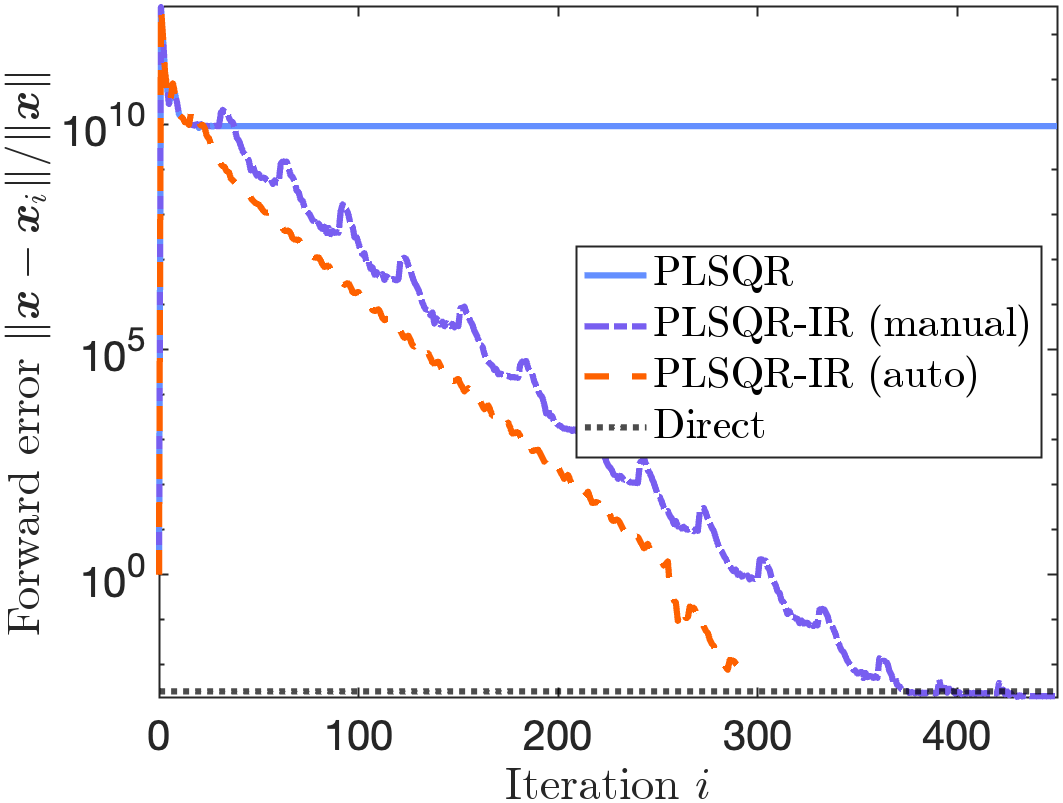}
    \caption{Backward (\emph{left}) and forward (\emph{right}) error PLSQR with no iterative refinement, iterative refinement every 15 iterations, and automatic iterative refinement}
    \label{fig:auto-lsqr-ir}
\end{figure}

Numerical demonstration of our PLSQR-IR implementation with automatic termination and iterative refinement is provided in \cref{fig:auto-lsqr-ir}.
We test on a harder version of the example in \cref{fig:lsqr} with $\cond(\mat{A}) = 10^{14}$ and $\cond(\mat{A}\mat{P}^{-1}) = 10$.
For this example, PLSQR without iterative refinement fails to reduce the backward error at all.
The purple dash-dotted curves show PLSQR-IR with manual iterative refinement every 15 iterations, and the orange dashed curve shows PLSQR-IR with automatic termination and iterative refinement.
The automatic procedure is seen to attain a machine-precision backward error faster than iterative refinement every 15 iterations.
(We note that a better manual choice of restart frequency can lead to comparable or even slightly faster convergence than the automatic restarting procedure on this example. 
Our point is merely that a suboptimal choice of the restart frequency can delay convergence and that the automatic restarting procedure yields good results without need for parameter tuning.)

\section{Analysis} \label{sec:analysis}

This section contains our analysis of iterative methods on the preconditioned normal equations with iterative refinement.
We will establish a formal and more general version of \cref{infthm:main} from the introduction.

\subsection{Philosophy, notation, and assumptions} \label{sec:notation}
The generally accepted definition is that a linear system solver is backward stable if \cref{eq:backward-stable} holds with $c$ equal to a small, low-degree polynomial in the matrix dimension $n$ \cite[sec.~7.6]{Hig02}.
Our analysis will primarily target iterative methods that converge geometrically fast, which we demonstrate to require just $q = \order(\log(1/u))$ iterations to converge to machine accuracy (even in finite-precision arithmetic). 
Under this setting, it is perhaps natural to permit the prefactor $c$ to depend polynomially on the iteration count $q$ in addition to the dimension $n$.
For this reason, this paper will employ the following definition:

\begin{definition}[Backward stability] 
    A linear system solver is \emph{backward stable} if there exists a prefactor $c$, polynomially large in $n$ and $\log(1/u)$, such that whenever $c\cond(\mat{A})u < 1$, the output $\vec{\hat{x}}$ of the procedure satisfies \cref{eq:backward-stable}.
\end{definition}
We believe this definition is reasonable, as factors of $n$ are larger than factors in $\log(1/u)$ for all but the smallest systems in double precision ($\log(1/u) < 37$).
It is generally understood that, for a method to be considered backward stable in a practical sense, the prefactor $c$ should a low-degree polynomial with modest-sized coefficients.

Our approach to backward error analysis in this paper is encapsulated in the following quote by backward error analysis pioneer Jim Wilkinson \cite[p.~356]{Wil74}:
\begin{quote}
    All too often, too much attention is paid to the precise error bound that has been established. The main purpose of [backward stability analysis] is either to establish the essential numerical stability of an algorithm or to show why it is unstable and in doing so to expose what sort of change is necessary to make it stable. The precise error bound is not of great importance.
\end{quote}
Our analysis will adopt Wilkinson's perspective in this quote, and we will not carefully track the prefactors in our analysis.
(As we will see, the proof is already complicated without such tracking.)
As the experiments in this paper demonstrate, the prefactors appear to be small in practice.

We will introduce notation and standing assumptions to simplify the proceedings.
We work under the standard model of floating point arithmetic \cite[sec.~2.2]{Hig02} with unit roundoff $u$.
For $a,b\ge 0$, we write $a\lesssim b$ to denote $a \le p(n,\log(1/u)) b$ for some polynomial function $p$,
and we write $a \ll b$ as a hypothesis to indicate $a \le b/p(n,\log(1/u))$ for a polynomial $p$ that can be taken as large as needed to complete the analysis.
We invoke the $\lesssim$ notation only when $b$ is a multiple of unit roundoff $u$.
We write $a \asymp b$ when $a$ and $b$ are comparable up to absolute constants.
The condition number of $\mat{A}$ is $\kappa \coloneqq \cond(\mat{A})$, and we assume throughout that $\mat{A}$ is numerically full rank, in the sense that
\begin{equation} \label{eq:numerically-full-rank}
    \kappa u \ll 1.
\end{equation}
We assume access to an excellent preconditioner $\mat{P}$ satisfying \cref{eq:good-prec}. 
To simplify the analysis, we assume that $\mat{A}$ and $\mat{P}$ are commensurately scaled:
\begin{equation} \label{eq:com-scale}
    \norm{\mat{A}\mat{P}^{-1}} \asymp 1.
\end{equation}
Combining \cref{eq:good-prec} and \cref{eq:com-scale} yields the useful relations
\begin{equation} \label{eq:AP-norms}
    \norm{\mat{P}} \asymp \norm{\mat{A}}, \quad \norm{\mat{P}^{-1}} \asymp \norm{\mat{A}^{-1}} = \frac{\kappa}{\norm{\mat{A}}}.
\end{equation}
The result of an expression $\vec{z}$ computed in floating point arithmetic will be denoted with either $\fl(\vec{z})$ or $\vec{\hat{z}}$, and the error is $\err(\vec{z}) \coloneqq \fl(\vec{z}) - \vec{z}$.
A vector is \emph{exactly representable if} $\fl(\vec{z}) = \vec{z}$.
Symbols $\vec{e},\vec{e}',\vec{e}_1,\vec{e}_2,\ldots$ denote some vector of norm $\lesssim u$.

\subsection{Main results}

\begin{algorithm}[t]
	\caption{Stable preconditioned linear system meta-solver} \label{alg:meta-algorithm}
	\begin{algorithmic}[1]
		\Require Matrix $\mat{A}\in\real^{n\times n}$, preconditioner $\mat{P} \in \real^{n\times n}$, vector $\vec{b}\in\real^n$, subroutine  $\Call{IterativeSolver}{}$, iterative refinment steps $t$
		\Ensure Approximate solution $\vec{x}_t\in\real^n$ to $\mat{A}\vec{x} = \vec{b}$
        \State $\vec{x}_0 \gets \vec{0}$
        \For{$i=0,1,\ldots,t-1$}
        \State $\vec{c}_i \gets \mat{P}^{-\top}(\mat{A}^\top(\vec{b} - \mat{A}\vec{x}_i))$
        \State $\vec{\delta y}_{i} \gets \Call{IterativeSolver}{\textsc{Apply} = \vec{z} \mapsto \mat{P}^{-\top}(\mat{A}^\top(\mat{A}(\mat{P}^{-1}\vec{z})),\textsc{RHS} = \vec{c}_i}$ 
        \State $\vec{x}_{i+1} \gets \vec{x}_i + \mat{P}^{-1}\vec{\delta y}_{i}$ \Comment{Iterative refinement}
        \EndFor
	\end{algorithmic}
\end{algorithm}

Our first main result treats the general ``meta-solver'' for linear systems by iteratively solving the preconditioned normal equations with iterative refinement from the introduction; see \cref{alg:meta-algorithm} for pseudocode.

\begin{theorem}[Backward stability for the meta-solver] \label{thm:main-meta}
    Let $\mat{A} \in \real^{n\times n}$ be numerically full-rank \cref{eq:numerically-full-rank} and assume that $\mat{P} \in \real^{n\times n}$ is a high-quality preconditioner \cref{eq:good-prec}.
    Assume that the operations $\mat{A}\vec{z}$, $\mat{A}^\top\vec{z}$, $\mat{P}^{-1}\vec{z}$, and $\mat{P}^{-\top}\vec{z}$ are performed in a backward stable way, and assume the \Call{IterativeSolver}{} routine produces an output $\vec{\hat{\delta y}}_i$ satisfying
    \begin{equation} \label{eq:iterative-solver-guarantee}
        \norm{\vec{\hat{\delta y}}_i - (\mat{P}^{-\top}\mat{A}^\top\mat{A}\mat{P}^{-1})^{-1}\vec{c}_i} \lesssim \kappa u \norm{\vec{c}_i}.
    \end{equation}
    The meta-solver \cref{alg:meta-algorithm} satisfies the error guarantee
    \begin{equation} \label{eq:meta-error-guarantee}
        \frac{\norm{\vec{b} - \mat{A}\vec{\hat{x}}_i}}{\norm{\mat{A}}\norm{\vec{x}}} \lesssim u + (c\kappa u)^i \quad \text{with } c\lesssim 1.
    \end{equation}
    In particular, $t = \order(\log(1/u) / \log(1/(\kappa u))$ iterative refinement steps suffice to attain a backward stable solution.
\end{theorem}

\begin{algorithm}[t]
	\caption{Lanczos method for linear solves $\mat{M}^{-1}\vec{c}$} \label{alg:lanczos}
	\begin{algorithmic}[1]
		\Require Matvec $\vec{z} \mapsto \mat{M}\vec{z}$ subroutine \Call{Apply}{}, right-hand side $\vec{c}$, iterations $q$
		\Ensure Approximate solution $\vec{y} \approx \mat{M}^{-1}\vec{c}$
		\State $\vec{q}_0 \gets \vec{0}$, $\gamma \gets \norm{\vec{c}}$, $\vec{q}_1\gets \vec{c}/\gamma$, $\beta_1 \gets 0$
        \For{$i = 1,\ldots,k$}
        \State $\vec{q}_{i+1} \gets \Call{Apply}{\vec{q}_i} - \beta_i \vec{q}_{i-1}$
        \State $\alpha_i \gets \vec{q}_{i+1}^\top \vec{q}_i^{\vphantom{\top}}$
        \State $\vec{q}_{i+1} \gets \vec{q}_{i+1} - \alpha_i \vec{q}_i$
        \State $\beta_{i+1} \gets \norm{\vec{q}_{i+1}}$
        \State \textbf{if} $\beta_{i+1} = 0$, set $k\gets i$ and \textbf{break}
        \State $\vec{q}_{i+1} \gets \vec{q}_{i+1} / \beta_{i+1}$
        \EndFor
        \State $\mat{T} \gets \begin{bmatrix} \alpha_1 & \beta_2 \\ 
        \beta_2 & \alpha_2 & \ddots \\
        &\ddots & \ddots & \beta_k \\ && \beta_k & \alpha_k\end{bmatrix}$, $\mat{Q} \gets \begin{bmatrix} \vec{q}_1 & \cdots & \vec{q}_k \end{bmatrix}$
        \State $\vec{y} \gets \gamma \cdot \mat{Q} (\mat{T}^{-1}\mathbf{e}_1)$
	\end{algorithmic}
\end{algorithm}

Existing techniques are not yet able to establish the guarantee \cref{eq:iterative-solver-guarantee} for the LSQR or CG methods.
To obtain a method with end-to-end guarantees, we may use the Lanczos procedure (\cref{alg:lanczos}) as the \Call{IterativeSolver}{} routine.
We remind the reader that LSQR, CG on the normal equations, and Lanczos on the normal equations are all equivalent in exact arithmetic \cite{Gre97a}.
The main disadvantage of Lanczos over CG and LSQR is storage cost; for $q$ iterations, standard implementations of Lanczos require $\order(qn)$ auxilliary storage, while CG and LSQR require $\order(n)$ auxilliary storage.
Let us emphasize that we are implementing Lanczos \emph{without reorthogonalization}, which is known to be stable for many tasks \cite[ch.~4]{Che24}.
For other uses of Lanczos, reorthogonalization may be necessary to remedy issues associated with loss of orthogonality (or obtain improved spectrum-adaptive convergence \cite[Fig.~5]{CHLZ25}).

\begin{theorem}[Backward stability for meta-solver with Lanczos] \label{thm:main-lanczos}
    Let $\mat{A}$ and $\mat{P}$ be as in \cref{thm:main-meta}, and execute the meta-solver \cref{alg:meta-algorithm} using the Lanczos algorithm \cref{alg:lanczos} as the linear solver with $q = \order(\log(1/(\kappa u)))$ iterations.
    With $t = \order(\log(1/u) / \log(1/(\kappa u))$ iterative refinement steps, \cref{alg:meta-algorithm} is backward stable.
    The total cost is $\order(\log(1/u))$ matvecs with $\mat{A}$, $\mat{A}^\top$, $\mat{P}^{-1}$, and $\mat{P}^{-\top}$ and $\order(n\log(1/u))$ additional arithmetic operations.
\end{theorem}

\subsection{Roadmap}
Our proof strategy draws on the approach and techniques developed by a subset of the present authors and Maike Meier in \cite{EMN24}.
We opt to provide a largely self-contained treatment, and our focus on square systems allows for dramatic simplification over corresponding results in \cite{EMN24}.

We begin by establishing \cref{thm:main-meta}.
The first core ingredient is the following lemma, which controls the error when we interleave multiplications with $\mat{A}$, its inverse-preconditioner $\mat{P}^{-1}$, and their adjoints.

\begin{lemma}[Interleaved multiplies with matrix and its preconditioner] \label{lem:interleaved}
    Let $\mat{A}$ and $\mat{P}$ be as in \cref{thm:main-meta}, and let $\vec{z}$ be exactly representable.
    Then
    \begin{equation}
        \bigl\|\err(\mat{P}^{-\top} (\mat{A}^\top \vec{z}))\bigr\|, \bigl\|\err(\mat{P}^{-\top} (\mat{A}^\top (\mat{A}(\mat{P}^{-1}\vec{z}))))\bigr\| \lesssim \kappa u \norm{\vec{z}}. \label{eq:interleaved}
    \end{equation}
\end{lemma}

We establish this result in \cref{sec:proof-interleaved}.
Using these bounds, we establish a recurrence for the error, whose proof appears in \cref{sec:proof-error-formula}.

\begin{lemma}[Single-step error bound] \label{lem:error-formula}
    Let $\mat{A}$ and $\mat{P}$ be as in \cref{thm:main-meta}.
    Then
    \begin{equation} \label{eq:r-bound}
        \norm{\vec{r}_{i+1}} \lesssim \kappa u \norm{\vec{r}_i} +\norm{\mat{A}}\norm{\vec{x}}u \quad \text{where $\vec{r}_i = \vec{b} - \mat{A}\vec{\hat{x}}_i$.}
    \end{equation}
\end{lemma}

Finally, using these tools, we establish \cref{thm:main-meta} in \cref{sec:proof-main-meta}.
The proof of \cref{thm:main-lanczos} follows from \cref{thm:main-meta} after establishing \cref{eq:iterative-solver-guarantee} for the Lanczos method, which is done in \cref{sec:proof-main-lanczos}.

\subsection{Proof of \cref{lem:interleaved}} \label{sec:proof-interleaved}

Before providing the formal proof of \cref{lem:interleaved}, we give a more intuitive explanation of using first-order perturbation theory.
Assuming backward stable operations, the product $\mat{P}^{-\top}\mat{A}^\top\mat{A}\mat{P}^{-1}\vec{z}$ is computed up to a backward error $\mat{\Delta}_i$ of machine-precision size $\norm{\mat{\Delta}_i} \lesssim \norm{\mat{A}}u$ to each occurence of $\mat{A}$ and $\mat{P}$.
We have the first-order expansion
\begin{multline*}
    (\mat{P}^\top + \mat{\Delta}_1)^{-1}(\mat{A}^\top + \mat{\Delta}_2)(\mat{A} + \mat{\Delta}_3)(\mat{P} + \mat{\Delta}_4)^{-1} \\ = \mat{M} - \mat{P}^{-\top} \mat{\Delta}_1 \mat{M} + \mat{P}^{-\top} \mat{\Delta}_2 \mat{A}\mat{P}^{-1} + \mat{P}^{-\top} \mat{A}^\top \mat{\Delta}_3 \mat{P}^{-1} - \mat{M}\mat{\Delta}_4\mat{P}^{-1} + \order(u^2).
\end{multline*}
By \cref{eq:com-scale,eq:AP-norms}, $\mat{M}$, $\mat{A}\mat{P}^{-1}$, and $\mat{P}^{-\top}\mat{A}^\top$ have norm $\lesssim 1$ and, $\mat{P}^{-1}$ and $\mat{P}^{-\top}$ have norm $\lesssim \kappa / \norm{\mat{A}}$.
Thus,
\begin{equation*}
    \norm{(\mat{P}^\top + \mat{\Delta}_1)^{-1}(\mat{A}^\top + \mat{\Delta}_2)(\mat{A} + \mat{\Delta}_3)(\mat{P} + \mat{\Delta}_4)^{-1} - \mat{M}} \lesssim \kappa u + \order(u^2);
\end{equation*}
the conclusion of \cref{lem:interleaved} is recovered, up to first order.
The interested reader may wish to repeat this analysis with the matrix
\begin{equation*}
    \mat{M}' = \mat{P}^{-1}\mat{P}^{-\top}\mat{A}^\top\mat{A},
\end{equation*}
which yields a larger error of size $\kappa^2 u$.

To prove \cref{lem:interleaved}, we first import stability results for vector arithmetic and matrix multiplication \cite[secs.~2--3]{Hig02}: 

\begin{fact}[Basic stability results] \label{fact:basic-stability}
    For exactly representable $\vec{w},\vec{z}\in\real^n$ and $\gamma \in \real$, we have
    \begin{equation*}
        \norm{\err(\vec{z} \pm \vec{w})}\lesssim \norm{\vec{z} \pm \vec{w}}u, \quad \norm{\err(\gamma\cdot \vec{z})} \lesssim |\gamma| \norm{\vec{z}} u.
    \end{equation*}
    For backward stable matrix multiplication by a matrix $\mat{B}$, we have
    \begin{equation*}
        \norm{\err(\mat{B}\vec{z})} \lesssim \norm{\mat{B}} \norm{\vec{z}} u, \quad \big\|\smash{\err(\mat{B}^\top\vec{z})}\big\| \lesssim \norm{\mat{B}} \norm{\vec{z}} u.
    \end{equation*}
\end{fact}

Our next result provides a useful result for backward stable linear solves:

\begin{proposition}[Backward stable solves] \label{prop:backward}
    Let $\mat{P}$ be as in \cref{thm:main-meta}, and let $\vec{z} \in \real^n$ be exactly representable.
    Then
    \begin{equation*}
        \err(\mat{P}^{-1}\vec{z}) = \norm{\mat{A}}\norm{\smash{\mat{P}^{-1}\vec{z}}} \cdot \mat{P}^{-1}\vec{e}, \quad \err(\mat{P}^{-\top}\vec{z}) = \norm{\mat{A}}\big\|\mat{P}^{-\top}\vec{z}\big\| \cdot \mat{P}^{-\top}\vec{e}'.
    \end{equation*}
    We remind the reader of the ``$\vec{e}$'' notation defined in \cref{sec:notation}.
\end{proposition}

\begin{proof}
    We prove the result for $\mat{P}^{-1}\vec{z}$; the result for $\mat{P}^{-\top}\vec{z}$ is identical.
    By hypothesis, solves $\mat{P}^{-1}\vec{z}$ are computed in a backward stable way:
    \begin{equation} \label{eq:P-backward}
        (\mat{P} + \mat{\Delta P})(\mat{P}^{-1}\vec{z} + \err(\mat{P}^{-1}\vec{z})) = \vec{z} \quad \text{with } \norm{\mat{\Delta P}} \lesssim \norm{\mat{A}}u.
    \end{equation}
    Here, we have used \cref{eq:AP-norms} to replace $\norm{\mat{P}}$ with $\norm{\mat{A}}$.
    Rearranging yields
    \begin{equation*}
        \err(\mat{P}^{-1}\vec{z})) = \mat{P}^{-1}\vec{w} \quad \text{for } \vec{w} = (\Id + \mat{\Delta P}\cdot \mat{P}^{-1})^{-1} (-\mat{\Delta P}) (\mat{P}^{-1}\vec{z}).
    \end{equation*}
    Establishing the bound requires proving $\norm{\vec{w}} \lesssim \norm{\mat{A}}\norm{\smash{\mat{P}^{-1}\vec{z}}}u$.
    To do so, apply \cref{eq:AP-norms,eq:P-backward} to bound
    \begin{equation*}
        \norm{\smash{\mat{\Delta P} \cdot \mat{P}^{-1}}} \le \norm{\mat{\Delta P}}\norm{\smash{\mat{P}^{-1}}} \lesssim \kappa u \ll 1 \quad \text{so } \norm{\smash{(\Id + \mat{\Delta P}\cdot \mat{P}^{-1})^{-1}}} \le 3/2.
    \end{equation*}
    Using this result and \cref{eq:P-backward} yields
    \begin{equation*}
        \norm{\vec{w}} \le \norm{\smash{(\Id + \mat{\Delta P}\cdot \mat{P}^{-1})^{-1}}} \cdot \norm{-\mat{\Delta P}} \cdot \norm{\smash{\mat{P}^{-1}\vec{z}}} \lesssim \norm{\mat{A}}\norm{\smash{\mat{P}^{-1}\vec{z}}}u.
    \end{equation*}
    The desired result is proven.
\end{proof}

With this result in hand, we prove \cref{lem:interleaved}.

\begin{proof}[Proof of \cref{lem:interleaved}]
    We establish the result for $\mat{P}^{-\top}\mat{A}^\top\vec{z}$  first.
    By \cref{fact:basic-stability},  
    \begin{equation} \label{eq:ATz-err}
        \big\|\smash{\err(\mat{A}^\top\vec{z})}\big\|\lesssim \norm{\mat{A}}\norm{\vec{z}}u.
    \end{equation}
    By \cref{prop:backward}, we have
    \begin{equation} \label{eq:err-PTATz}
        \err(\mat{P}^{-\top}\mat{A}^\top\vec{z}) = \mat{P}^{-\top}\err(\mat{A}^\top\vec{z}) + \norm{\mat{A}} \big\| \mat{P}^{-\top}\fl(\mat{A}^\top\vec{z})\big\| \cdot \mat{P}^{-\top}\vec{e}_1.
    \end{equation}
    The first term accounts for the action of $\mat{P}^{-\top}$ on the errors produced when computing $\mat{A}^\top\vec{z}$, and the second term accounts for the errors in evaluating $\mat{P}^{-\top}\fl(\mat{A}^\top\vec{z})$.
    First we bound $\big\| \mat{P}^{-\top}\fl(\mat{A}^\top\vec{z})\big\|$:
    \begin{align*}
        \big\| \mat{P}^{-\top}\fl(\mat{A}^\top\vec{z})\big\| &\le \big\| \mat{P}^{-\top}\mat{A}^\top\vec{z}\big\| +  \big\| \mat{P}^{-\top}\err(\mat{A}^\top\vec{z})\big\| \\&\lesssim \big\| \mat{P}^{-\top}\mat{A}^\top\big\|\norm{\vec{z}} + \big\| \mat{P}^{-\top}\big\| \norm{\mat{A}}\norm{\vec{z}}u \lesssim \norm{\vec{z}}.
    \end{align*}
    The second inequality is \cref{eq:ATz-err} and the submultiplicative property.
    For the third inequality, we note that first term is $\lesssim \norm{\vec{z}}$ by \cref{eq:com-scale} and the second term is $\lesssim \kappa u \norm{\vec{z}} \lesssim \norm{\vec{z}}$ by \cref{eq:AP-norms,eq:numerically-full-rank}.
    The bound $\bigl\|\err(\mat{P}^{-\top}\mat{A}^\top\vec{z})\bigr\| \lesssim \kappa u \norm{\vec{z}}$ now follows by taking the norm of \cref{eq:err-PTATz}, invoking the triangle inequality and submultiplicative properties, and using \cref{eq:AP-norms}.
    We have established the first estimate in \cref{eq:interleaved}.

    A similar argument shows
    \begin{equation} \label{eq:err-APz}
        \|{\err(\mat{A}\mat{P}^{-1}\vec{z})}\| \lesssim \kappa u \norm{\vec{z}}.
    \end{equation}
    Now, we chain the two bounds together.
    First, obtain the norm bound:
    \begin{equation}\label{eq:norm-APz}
        \|{\fl(\mat{A}\mat{P}^{-1}\vec{z})}\| \le \|\mat{A}\mat{P}^{-1}\vec{z}\| + \|{\err(\mat{A}\mat{P}^{-1}\vec{z})}\| \lesssim \norm{\vec{z}} + \kappa u \norm{\vec{z}} \lesssim \norm{\vec{z}}.
    \end{equation}
    The second inequality is \cref{eq:com-scale,eq:err-APz}, and the third inequality is the hypothesis $\kappa u \ll 1$.
    By the first estimate in \cref{eq:interleaved}, we have
    \begin{equation*}
        \big\| {\err(\mat{P}^{-\top}\mat{A}^\top\mat{A}\mat{P}^{-1}\vec{z})} \big\| \lesssim \big\| \mat{A}^\top\mat{P}^{-\top} \err(\mat{A}\mat{P}^{-1}\vec{z})\big\| + \kappa u \|{\fl(\mat{A}\mat{P}^{-1}\vec{z})}\|.
    \end{equation*}
    Invoking the submultiplicative property, \cref{eq:com-scale}, \cref{eq:err-APz}, and \cref{eq:norm-APz} yields the second estimate in \cref{eq:interleaved}.
\end{proof}

\subsection{Proof of \cref{lem:error-formula}} \label{sec:proof-error-formula}

For notational convenience, the functions $\err$, $\fl$, etc.\ report only the numerical errors computed during the $i$th step.
Denote the residual $\vec{r}_i \coloneqq \vec{b} - \mat{A}\vec{\hat{x}}_i$.
By \cref{fact:basic-stability}, $\vec{r}_i$ is computed up to error
\begin{equation*}
    \norm{\err(\vec{r}_i)} \lesssim (\norm{\vec{r}_i} + \norm{\mat{A}}\norm{\vec{\hat{x}}_i}) u \quad \text{so } \norm{\fl(\vec{r}_i)} \le 2\norm{\vec{r}_i}+ c\norm{\mat{A}}\norm{\vec{\hat{x}}_i} u 
\end{equation*}
for a prefactor $c = p(n,\log(1/u))$.
Therefore, by \cref{lem:interleaved}, the error in $\vec{c}_i$ in Algorithm \ref{alg:meta-algorithm} satisfies
\begin{equation} \label{eq:c-err}
    \begin{split}
    \norm{\err(\vec{c}_i)} &\lesssim \norm{\mat{P}^{-\top}\mat{A}^\top \err(\vec{r}_i)} + \kappa u \norm{\mat{A}}\norm{\fl(\vec{r_i})} \\
    &\lesssim \kappa u \norm{\vec{r}_i} + \norm{\mat{A}}\norm{\vec{\hat{x}}_i}u.
    \end{split}
\end{equation}
In the second line, we invoke the previous display and use the bounds \cref{eq:com-scale,eq:numerically-full-rank}.
Using the hypothesis $\kappa u \ll 1$, the error formula \cref{eq:c-err} leads to the bound
\begin{equation}\label{eq:c-norm}
    \norm{\fl(\vec{c}_i)} \le \norm{\vec{c}_i} +  \norm{\err(\vec{c}_i)} \lesssim \norm{\vec{c}_i} + \kappa u\norm{\vec{r}_i} + \norm{\mat{A}}\norm{\vec{\hat{x}}_i} u .
\end{equation}

Now we turn to the computation of $\vec{\delta y}_i$, defined as
\begin{equation} \label{eq:dy}
    \vec{\delta y}_i = \mat{M}^{-1}\vec{c}_i = \mat{P}\mat{A}^{-1}\vec{r}_i \quad \text{for } \mat{M} \coloneqq \mat{P}^{-\top}\mat{A}^\top\mat{A}\mat{P}^{-1}.
\end{equation}
Under the guarantee \cref{eq:iterative-solver-guarantee} for the \Call{IterativeSolver}{} routine, the numerically computed output $\vec{\hat{\delta y}}_i$ has error
\begin{equation*}
    \norm{\err(\vec{\delta y}_i)} \lesssim \norm{\mat{M}^{-1}\err(\vec{c}_i)} + \kappa u \norm{\fl(\vec{c}_i)}.
\end{equation*}
The first term comes from applying the exact inverse $\mat{M}^{-1}$ to the errors in $\vec{c}_i$, and the second term comes from the inexact nature of the \Call{IterativeSolver}{} routine.
By \cref{eq:good-prec,eq:com-scale}, the matrix $\mat{M}^{-1}$ has norm $\lesssim 1$.
Ergo, using the formulas \cref{eq:c-err,eq:c-norm} and the fact that $\norm{\vec{c}_i} \le K\norm{\vec{r}_i}$ for a constant $K>0$, we obtain
\begin{equation} \label{eq:err-dy}
    \norm{\err(\vec{\delta y}_i)} \lesssim \kappa u \norm{\vec{r}_i} + \norm{\mat{A}}\norm{\vec{\hat{x}}_i}u.
\end{equation}
We have again used \cref{eq:numerically-full-rank} to simplify.

Now we treat the errors in evaluating $\mat{P}^{-1}\vec{\delta y}_i$. 
By \cref{prop:backward}, we have
\begin{equation} \label{eq:Pinvy}
    \err(\mat{P}^{-1}\vec{\delta y}_i) = \mat{P}^{-1} \err(\vec{\delta y}_i) + \norm{\mat{A}}\|\mat{P}^{-1}\fl(\vec{\delta y}_i)\| \cdot \mat{P}^{-1}\vec{e}_1.
\end{equation}
Let us begin by bounding $\norm{\smash{\mat{P}^{-1}\fl(\vec{\delta y}_i)}}$.
Observe that \cref{eq:dy} implies $\mat{P}^{-1}\vec{\delta y}_i = \mat{A}^{-1}(\vec{b} - \mat{A}\vec{\hat{x}}_i) = \vec{x} - \vec{\hat{x}}_i$.
Using the formulas \cref{eq:dy,eq:err-dy}, we may write
\begin{align*}
    \norm{\smash{\mat{P}^{-1}\fl(\vec{\delta y}_i)}} &\le \norm{\vec{x} - \vec{\hat{x}}_i} + \|\mat{P}^{-1}\|\norm{\err(\vec{\delta y}_i)} \lesssim \norm{\vec{x}} + \norm{\vec{x}-\vec{\hat{x}}_i} + \kappa \norm{\vec{r}_i}/\norm{\mat{A}}.
\end{align*}
In the second inequality, we used \cref{eq:AP-norms}, \cref{eq:numerically-full-rank}, and the the triangle inequality.
Consequently, substituting in \cref{eq:Pinvy} and using \cref{eq:err-dy}, we obtain
\begin{equation*}
    \err(\mat{P}^{-1}\vec{\delta y}_i) = (\kappa \norm{\vec{r}_i} + \norm{\mat{A}}\norm{\vec{x}} + \norm{\mat{A}}\norm{\vec{x} - \vec{\hat{x}}_i})\cdot \mat{P}^{-1}\vec{e}_2.
\end{equation*}
We can clean this bound up by noting
\begin{equation*}
    \norm{\mat{A}}\norm{\vec{x} - \vec{\hat{x}}_i} \le \norm{\mat{A}}\|\mat{A}^{-1}\|\norm{\mat{A}(\vec{x} - \vec{\hat{x}}_i)} = \kappa \norm{\vec{r}_i}.
\end{equation*}
Using this bound and invoking \cref{fact:basic-stability} yields
\begin{equation*}
    \vec{x} - \vec{\hat{x}}_{i+1} = (\kappa \norm{\vec{r}_i} + \norm{\mat{A}}\norm{\vec{x}})\cdot\mat{P}^{-1}\vec{e}_3. 
\end{equation*}
Multiplying by $\mat{A}$, taking norms, and using \cref{eq:com-scale} gives the desired result.
\hfill $\proofbox$

\subsection{Proof of \cref{thm:main-meta}} \label{sec:proof-main-meta}

Equation~\cref{eq:r-bound} gives a recurrence for the residual norm with initial condition $\norm{\vec{r}_0} = \norm{\vec{b}} \le \norm{\mat{A}}\norm{\vec{x}}$.
Solving it yields the bound
\begin{equation*}
    \norm{\vec{r}_i} \lesssim (c\kappa u)^i\norm{\mat{A}}\norm{\vec{x}} + \frac{c\norm{\mat{A}}\norm{\vec{x}}u}{1-c\kappa u}.
\end{equation*}
The prefactor $c$ is a polynomial in $n$.
Using the hypothesis $\kappa u \ll 1$ and rearranging gives the claimed bound \cref{eq:meta-error-guarantee}.

To conclude the stated backward stability claim, recall from \cref{eq:rigal-gaches} that the backward error $\mat{\Delta A}_i$ for the $i$th approximate solution $\vec{\hat{x}}_i$ is 
\begin{equation} \label{eq:rigal-gaches-conclusion}
    \frac{\norm{\mat{\Delta A}_i}}{\norm{\mat{A}}} = \frac{\norm{\vec{b} - \mat{A}\vec{\hat{x}}_i}}{\norm{\mat{A}}\norm{\vec{\hat{x}}_i}} \le \frac{\norm{\vec{b} - \mat{A}\vec{\hat{x}}_i}}{\norm{\mat{A}}(\norm{\vec{x}} - \norm{\vec{x}-\vec{\hat{x}}_i})}.
\end{equation}
The inequality holds provided the denominator is positive.
From \cref{eq:meta-error-guarantee}, we have
\begin{equation*}
    \norm{\vec{x}-\vec{\hat{x}}_i} \le \kappa [c'u + (c'\kappa u)^i]\norm{\mat{x}}.
\end{equation*}
The prefactor $c'$ is a polynomial in $n$.
Using the assumption $\kappa u \ll 1$, it follows that after $t = \order(\log(1/u)/\log(1/(\kappa u)))$ iterations, we have 
\begin{equation*}
    \norm{\vec{x}-\vec{\hat{x}}_t} \le \norm{\vec{x}}/2 \quad \text{and} \quad \norm{\vec{b} - \mat{A}\vec{\hat{x}}_t} \lesssim \norm{\mat{A}}\norm{\vec{x}} u.
\end{equation*}
Substituting these bounds into \cref{eq:rigal-gaches-conclusion} establishes the bound $\norm{\mat{\Delta A}_i}/\norm{\mat{A}}\lesssim u$, completing the proof. \hfill $\proofbox$

\subsection{Proof of \cref{thm:main-lanczos}} \label{sec:proof-main-lanczos}

To prove \cref{thm:main-lanczos}, we must establish that the Lanczos method (\cref{alg:lanczos}) produces a solution $\vec{\hat{\delta y}}$ to the system
\begin{equation*}
    \mat{M}\,\vec{\delta y} = \vec{c} \quad \text{with } \mat{M} \coloneqq \mat{P}^{\-\top}\mat{A}^\top\mat{A}\mat{P}^{-1}
\end{equation*}
satisfying $\norm{\smash{\vec{\hat{\delta y}} - \vec{\delta y}}} \lesssim \kappa u \norm{\vec{c}}$.
We shall use the symbol $\mat{M}$ throughout this subsection.
Recall from \cref{eq:good-prec,eq:com-scale} that $\norm{\mat{M}} \asymp 1$ and $\cond(\mat{M}) \le C^2$.

The first step will be to reinterpret \cref{eq:interleaved} from \cref{lem:interleaved} in the following way:
\actionbox{Computing $\mat{M}\vec{z}$ as $\mat{P}^{\-\top}(\mat{A}^\top(\mat{A}(\mat{P}^{-1}\vec{z})))$ attains the standard stability guarantee for matrix multiplication by $\mat{M}$ (\cref{fact:basic-stability}), but with a \emph{lower precision}:
\begin{equation}\label{eq:lanczos-apply}
    \norm{\fl(\mat{P}^{\-\top}(\mat{A}^\top(\mat{A}(\mat{P}^{-1}\vec{z})))) - \mat{M}\vec{z}} \lesssim \norm{\mat{M}} \norm{\vec{z}} \tilde{u} \quad \text{with $\tilde{u} \lesssim \kappa u$.}
\end{equation}}
\noindent This observation shows that the rounding errors incurred by executing the Lanczos procedure on the matrix $\mat{M}$ represented implicitly as the product $\mat{M} = \mat{P}^{\-\top}\mat{A}^\top\mat{A}\mat{P}^{-1}$ are no worse worse than executing the Lanczos procedure on the directly stored matrix $\mat{M}$ in a lower precision $\tilde{u} \lesssim \kappa u$.
We have the following result:

\begin{fact}[Lanczos: Well-conditioned input] \label{fact:lanczos}
    Assume $\mat{M}$ is well-conditioned, $\cond(\mat{M}) \le C^2$ for an absolute constant $C > 0$, and assume that \Call{Apply}{} satisfies
    \begin{equation} \label{eq:lanczos-hypothesis}
        \norm{\fl(\Call{Apply}{\vec{z}}) - \mat{M}\vec{z}} \lesssim \norm{\mat{M}}\norm{\vec{z}} \tilde{u} \quad \text{for all }\vec{z} \in \real^n
    \end{equation}
    for some effective precision $\tilde{u} > 0$.
    The Lanczos algorithm (\cref{alg:lanczos}) run for $q = \order(\log(1/\tilde{u}))$ iterations produces an output $\vec{\hat{y}}$ satisfying
    \begin{equation} \label{eq:lanczos-conclusion}
        \norm{\mat{M}} \cdot \norm{\vec{\hat{y}} - \mat{M}^{-1}\vec{c}} \le p(\log(1/\tilde{u})) \norm{\vec{c}}\tilde{u}
    \end{equation}
    for a polynomial $p$.
\end{fact}
This result essentially appears in \cite[sec.~6]{MMS18}; the derivation of the version of this result statement here is discussed in \cite[app.~C]{EMN24}.
We note that while the paper \cite{MMS18} establishes the \emph{qualitative} conclusion \cref{eq:lanczos-conclusion} that we want, other analyses \cite{DK92,DGK98} of the Lanczos procedure achieve smaller prefactors but assume that line 11 of \cref{alg:lanczos} is executed without error.
(This assumption should be straightforward to remove with some work.)
The techniques of \cite{MMS18} yield this result with $p(\log(1/\tilde{u})) = \order(\log^4(1/\tilde{u}))$; in ongoing work of the first two authors, Deeksha Adil, and Christopher Musco, we have improved the prefactor to $p(\log(1/\tilde{u})) = \order(\log^{1.5}(1/\tilde{u}))$.

We emphasize that this result holds \emph{only} when the matrix $\mat{M}$ is well-conditioned, and the existing results exhibit a poor dependence on the condition number of $\mat{M}$.
This existing analysis of the Lanczos method, with all its prefactors, appear to overestimate the errors that the Lanczos method incurs in practice by orders of magnitude, and---were you to carry out a sharp version of our argument tracking the prefactors---our bounds could be quantitatively uninteresting for practical values of $u$, $\kappa$, $C$, and $n$.
We are optimistic that, should a fuller and more quantitatively sharp understanding of the Lanczos method in finite precision be developed, plugging such new analysis into our bounds would give more revealing answers.

With this result in hand, we can prove \cref{thm:main-lanczos}.

\begin{proof}[Proof of \cref{thm:main-lanczos}]
    By \cref{eq:lanczos-apply}, the apply operation $\vec{z} \mapsto \mat{P}^{-\top}(\mat{A}^\top(\mat{A}(\mat{P}^{-1}\vec{z})))$ satisfies the hypothesis \cref{eq:lanczos-hypothesis} of \cref{fact:lanczos}.
    Ergo, the conclusion \cref{eq:lanczos-conclusion} holds, from whence the Lanczos method satisfies the necessary condition \cref{eq:iterative-solver-guarantee} for the \Call{IterativeSolver}{} method required to invoke \cref{thm:main-meta}.
    Invoking \cref{thm:main-meta} completes the proof.
\end{proof}

\section*{Acknowledgments}
We thank Deeksha Adil, Maike Meier, Lorenzo Lazzarino, Christopher Musco, Joel Tropp, and Umberto Zerbinati for valuable discussions.

\appendix

\section{Extra details for experiments} \label{sec:experiment-details}
Experiments are performed in MATLAB R2023b on a Macbook Pro with an M3 Pro chip and 18 GB of unified memory.
With the exception of the large example in the right panel of \cref{fig:pd}, residuals for all experiments are computed in higher precision using the \texttt{vpa} command in MATLAB with 24 digits.

\paragraph{\Cref{fig:lsqr,fig:left-lsqr,fig:auto-lsqr-ir}}
Test matrices $\mat{A} = \mat{U}\mat{\Sigma}\mat{V}^\top$ of dimension $n=1000$ are constructed by generating Haar-random orthogonal matrices and a diagonal matrix with logarithmically spaced entries $\mat{\Sigma} = \diag(10^{-\alpha(i-1)})$.
The preconditioner is given by $\mat{P}^{-1} = \mat{V}\mat{\tilde{\Sigma}}$ (\cref{fig:lsqr,fig:auto-lsqr-ir}) or $\mat{P}^{-1} = \smash{\mat{\tilde{\Sigma}}\mat{U}^\top}$ (\cref{fig:left-lsqr}), where $\mat{\tilde{\Sigma}}$ is a diagonal matrix constructed so that the entries of $\smash{\mat{\tilde{\Sigma}}}^{-1}\mat{\Sigma} = \diag(1 + \beta(i-1))$ are equally spaced.
The constants $\alpha,\beta$ is chosen to give the specified condition numbers for $\mat{A}$ and $\mat{A}\mat{P}^{-1}$ or $\mat{P}^{-1}\mat{A}$.
The inverse preconditioner $\mat{P}^{-1}$ is stored as a dense array.

\paragraph{\Cref{fig:pd}}
For the left panel of \cref{fig:pd}, we generate a random matrix $\mat{A} = \mat{U}\mat{\Sigma}\mat{U}^\top$ of dimension $n=1000$, where $\mat{U}$ is a Haar-random orthogonal matrix and $\mat{\Sigma} = \diag(10^{-14(i-1)/(n-1)}$ has logarithmically spaced eigenvalues between $10^{-14}$ and $10^0$.
We generate a preconditioner $\mat{P} = \mat{A} = \smash{\mat{U}\mat{\Sigma}^{1/2}\mat{W}\mat{\Sigma}^{1/2}\mat{U}^\top}$ by introducing a Wishart matrix $\mat{W} = \mat{G}^\top\mat{G}$, constructed from a $4n\times n$ matrix $\mat{G}$ with independent standard Gaussian matrix.
This is an artificial preconditioner designed to illustrate the behavior of preconditioned CG.
The matrix $\mat{P}$ is stored as a dense array, its Cholesky factorization is computed $\mat{P} = \mat{R}^\top \mat{R}$, and the inverse-preconditioner $\mat{P}^{-1}\vec{y} = \mat{R}^{-\top}(\mat{R}^{-1}\vec{y})$ is applied using triangular substitution.

In the right panel of \cref{fig:pd}, we generate a square-exponential kernel matrix $\mat{K} = (\exp(-\norm{\vec{x}_i - \vec{x}_j}^2/2))_{i,j=1}^n$ from a uniformly random subsample of $n=10^4$ points from the \texttt{COMET\_MC\_SAMPLE} dataset from the LIBSVM repository \cite{chang2011libsvm}.
Data is standardized so that each feature has mean-zero and variance one.
We solve the \emph{regularized} linear system $\mat{A}\vec{x} = \vec{b}$ for $\mat{A} = \mat{K} + \lambda\Id$ and $\lambda = 10^{-10}$.
To construct a preconditioner, we run the randomly pivoted Cholesky algorithm \cite{CETW25} for $k = 850$ steps to obtain a low-rank approximation $\mat{F}\mat{F}^\top \approx \mat{A}$ with $\mat{F} \in \real^{n\times k}$.
We utilize the Nystr\"om preconditioner \cite{FTU23,DEF+23} $\mat{P} = \mat{F}\mat{F}^\top + \lambda \Id$, which we apply via the formula
\begin{equation*}
    \mat{P}^{-1}\vec{y} = \mat{U} ( (\mat{\Sigma}^2 + \lambda \Id)^{-1}(\mat{U}^\top\vec{y})) + \frac{1}{\lambda}(\vec{y} - \mat{U}(\mat{U}^\top\vec{y})),
\end{equation*}
where $\mat{F} = \mat{U}\mat{\Sigma}\mat{V}^\top$ is a economy-size SVD.
Theoretical analysis and extensive testing for this \emph{randomly pivoted Cholesky preconditioning} approach are provided in \cite{DEF+23}.

\section{A critical analysis of GMRES's strengths} \label{sec:gmres-strengths}
There are a number of reasons for GMRES's popularity over LSQR for nonsymmetric linear systems.
In this section, we review the benefits of GMRES, providing somewhat of a counterpoint to the pro-LSQR argument made in \cref{sec:lsqr}.

\myparagraph{1.\ Adjoint-free.}
The (P)GMRES method requires only multiplications with the matrix $\mat{A}\vec{z}$ and inverse-preconditioner $\mat{P}^{-1}\vec{z}$.
LSQR also requires the adjoint operations $\mat{A}^\top\vec{z}$ and inverse-preconditioner $\mat{P}^{-\top}\vec{z}$.
For applications where adjoint operations are not available, GMRES may be the only choice.

\myparagraph{2.\ Numerical effects.}
GMRES is widely believed to have excellent stability properties.
Indeed, both the Householder \cite{DGRS95} and the modified Gram--Schmidt \cite{GRS97} versions of GMRES have been shown to be backward stable.
By contrast, there seem to be lingering concerns about the numerical properties of LSQR in finite precision arithmetic, perhaps based on its connection to the normal equations.
As this paper demonstrates, this skepticism is, in some ways, well-founded, but is easily cured using iterative refinement (at least in the most practically useful case in which one has a good preconditioner).
Conversely, we are not aware of any analysis that demonstrates that \emph{preconditioned} GMRES produces backward stable solutions in a small number of iterations.

Another virtue of GMRES is that, because it orthonormalizes each new Krylov basis vector against all previous basis vectors, it can attain more rapid convergence in finite precision than methods like CG and LSQR based on three term recurrences.
See the very recent paper \cite[Fig.~2]{DNRY25} for striking examples of this effect, where GMRES outperforms CG by significant measures.
These issues with CG and LSQR can be cured by implementing these methods with full or partial reorthogonalization, though they bring the storage and arithmetic costs of CG and LSQR closer to GMRES.

\myparagraph{3.\ Convergence rate.}
As is demonstrated in the left and right panels of \cref{fig:gmres-lsqr}, when the eigenvalues of GMRES are real and clustered and the eigenvectors are well-conditioned, GMRES converges at the ``CG rate'' of roughly $\e^{-i/\sqrt{\cond(\mat{A}\mat{P}^{-1})}}$, whereas LSQR converges at the slower rate  $\e^{-i/\cond(\mat{A}\mat{P}^{-1})}$.
Thus, \emph{provided one can cluster the eigenvalues and maintain well-conditioned eigenvectors}, GMRES converges faster and is more robust to a middling-quality preconditioner than LSQR is.

Another class of examples where GMRES \emph{often} achieves better convergence than LSQR is when the eigenvalues are clustered but the eigenvectors are ill-conditioned.
In principle, the bound \cref{eq:gmres-conv} suggests GMRES should also be negatively affected by ill-conditioned eigenvectors, but there are examples where (P)GMRES does well in spite of this \cite[ex.~2]{Wat22}; see \cite[sec.~3]{NRT92} for a partial theoretical explanation.

\bibliographystyle{siamplain}
\bibliography{short}

\end{document}